\documentclass[a4paper,11pt]{amsart}
\usepackage[a4paper]{geometry}

\usepackage{amsmath,amsthm}
\usepackage{amssymb,amsfonts}
\usepackage[hidelinks]{hyperref}
\usepackage{indentfirst}
\usepackage{tikz}
\usepackage{epigraph}
\usepackage{enumitem} 
\usepackage{graphicx}
\usepackage{calligra}
\usepackage{stmaryrd}
\usepackage{chngpage}
\usepackage{tikz-cd}
\usepackage{mathtools}
\DeclareSymbolFont{bbold}{U}{bbold}{m}{n}
\DeclareSymbolFontAlphabet{\mathbbm}{bbold}

\theoremstyle{plain}
	\newtheorem{theorem*}{Theorem}
	\newtheorem{theorem}{Theorem}[section]
	\newtheorem{theoremi}{Theorem}
	\newtheorem{corollaryi}[theoremi]{Corollary}
\numberwithin{equation}{section}

	\newtheorem{proposition}[theorem]{Proposition}
	\newtheorem{lemma}[theorem]{Lemma}
	\newtheorem{corollary}[theorem]{Corollary}
	\newtheorem{claim}{Claim}[theorem]

\theoremstyle{definition}
	\newtheorem{definition}[theorem]{Definition}
	\newtheorem{notation}[theorem]{Notation}

	\newtheorem{question}[theorem]{Question}
	\newtheorem*{acknow}{Acknowledgements}
\theoremstyle{remark}
	\newtheorem{remark}[theorem]{Remark}

\numberwithin{equation}{section}

\newcommand{\EFD}{\text{EFD}}
\newcommand{\PI}{\text{PI}}
\newcommand{\rk}{\text{qr}}

\newcommand{\Clop}{\text{Clop}}

\newcommand{\bbZ}{\mathbb{Z}}
\newcommand{\bbQ}{\mathbb{Q}}
\newcommand{\N}{\mathbb{N}}

\newcommand{\cstar}{$\mathrm{C}^\ast$}

\newcommand{\bbN}{\mathbb{N}}
\newcommand{\e}{\varepsilon}

\newcommand{\xMapsto}[2][]{\ext@arrow 0599{\Mapstofill@}{#1}{#2}}
\def\Mapstofill@{\arrowfill@{\Mapstochar\Relbar}\Relbar\Rightarrow}
\makeatother

\newcommand{\dotminus}{ 
\buildrel\textstyle\ .\over{\hbox{ 
\vrule height3pt depth0pt width0pt}{\smash-} 
}\ }

\newcommand{\norm}[1]{\left\lVert #1 \right\rVert}

\renewcommand{\phi}{\varphi}

\title{Games on AF-algebras}
\date{\today}

\author[BdB]{Ben De Bondt}
\address[Ben De Bondt]
{Institut de Math\'ematiques de Jussieu (IMJ-PRG)\\
Universit\'e Paris Cit\'e\\
B\^atiment Sophie Germain\\
8 Place Aur\'elie Nemours \\ 75013 Paris, France}
\email{ben.de-bondt@imj-prg.fr}
\urladdr{https://perso.imj-prg.fr/ben-debondt/}

\author[AV]{Andrea Vaccaro}
\address[Andrea Vaccaro]
{Institut de Math\'ematiques de Jussieu (IMJ-PRG)\\
Universit\'e Paris Cit\'e\\
B\^atiment Sophie Germain\\
8 Place Aur\'elie Nemours \\ 75013 Paris, France}
\email{vaccaro@imj-prg.fr}
\urladdr{https://sites.google.com/view/avaccaro/home}

\author[BV]{Boban Veli\v{c}kovi\'c}
\address[Boban Veli\v{c}kovi\'c]
{Institut de Math\'ematiques de Jussieu (IMJ-PRG)\\
Universit\'e Paris Cit\'e\\
B\^atiment Sophie Germain\\
8 Place Aur\'elie Nemours \\ 75013 Paris, France}

\author[AV]{Alessandro Vignati}
\address[Alessandro Vignati]
{Institut de Math\'ematiques de Jussieu (IMJ-PRG)\\
Universit\'e Paris Cit\'e\\
B\^atiment Sophie Germain\\
8 Place Aur\'elie Nemours \\ 75013 Paris, France}
\email{alessandro.vignati@imj-prg.fr}
\urladdr{http://www.automorph.net/avignati}



\begin{document}
\maketitle
\setcounter{tocdepth}{1}
\begin{abstract}
We analyze \cstar-algebras, particularly AF-algebras, and their $K_0$-groups in the context of the infinitary logic $\mathcal{L}_{\omega_1 \omega}$. Given two separable unital AF-algebras $A$ and $B$, and considering their $K_0$-groups as ordered unital groups, we prove that $K_0(A) \equiv_{\omega \cdot \alpha} K_0(B)$ implies $A \equiv_\alpha B$, where $M \equiv_\beta N$ means that $M$ and $N$ agree on all sentences of quantifier rank at most $\beta$. This implication is proved using techniques from Elliott's classification of separable AF-algebras, together with an adaptation of the Ehrenfeucht-Fra\"iss\'e game to
the metric setting. We use moreover this result to build a family $\{ A_\alpha \}_{\alpha < \omega_1}$ of pairwise non-isomorphic separable simple unital AF-algebras which satisfy $A_\alpha \equiv_\alpha A_\beta$ for every $\alpha < \beta$. In particular, we obtain a set of separable simple unital AF-algebras of arbitrarily high Scott rank. Next, we give a partial converse to the aforementioned implication, showing that $A \otimes \mathcal{K} \equiv_{\omega + 2 \cdot \alpha +2} B \otimes \mathcal{K}$ implies $K_0(A) \equiv_\alpha K_0(B)$, for every unital \cstar-algebras $A$ and $B$.
\end{abstract}

\section{Introduction}\label{S.Intro}

In this paper we investigate approximately finite \cstar-algebras, or simply AF-algebras, in the setting of infinitary continuous logic. 

Infinitary logics are extensions of first order logic that allow infinite expressions as formulas. In the classical discrete setting there are many instances of such logics that are still tame enough to allow a useful model theory to be developed (see e.g. \cite{Barwise-Feferman}). One of the most basic examples is $\mathcal L_{\omega_1\omega}$ which, in addition to the usual formulas considered in first order logic, also allows countable conjunctions and disjunctions, 
subject only to the restriction that the total number of free variables remains finite.

The adaptation of first order model theory to the continuous setting is by now a well-established field of research (we refer to \cite{BYBHU} for an introduction to this topic). A version of $\mathcal L_{\omega_1\omega}$ suitable for metric structures was introduced in \cite{BYI.MTForcing} (see also \cite{Eagle.Omit}). This work was later expanded in \cite{SRmetric}, where a large portion of Scott analysis for discrete infinitary logic has been extended to the metric setting. A basic definition in this framework, which has an immediate generalization to continuous logic, is the notion of \emph{quantifier rank}. The quantifier rank of an $\mathcal{L}_{\omega_1 \omega}$-formula $\phi$ is a countable ordinal $\rk(\phi)$ quantifying the complexity of $\phi$ depending on the length of chains of nested quantifiers appearing in it (see \S\ref{ss.dil} and \S\ref{sss.metricinf}). This rank naturally defines a hierarchy on $\mathcal{L}_{\omega_1 \omega}$-formulas, which is reflected on structures through the relations $\equiv_\alpha$. Given two structures $M$ and $N$ in the same (discrete or metric) language $\mathcal{L}$, and an ordinal $\alpha < \omega_1$, we write $M \equiv_\alpha N$ if $M$ and $N$ agree on all sentences of quantifier rank at
most $\alpha$. The main goal of this paper is to investigate the behavior of the relations $\equiv_\alpha$ in the context of \cstar-algebras and of their $K_0$-groups, focusing in particular on AF-algebras.

For the purpose of this paper, an AF-algebra is a {separable} \cstar-algebra $A$ containing a directed family of finite-dimensional subalgebras whose union is dense in $A$. The separability assumption on $A$ allows us to always assume that such directed families are indexed over $\bbN$. The definition of AF-algebra gives rise to an extremely rich family of (infinite dimensional) \cstar-algebras, which yet satisfy several robust and desirable regularity properties. Exploiting such structural features, in his seminal work \cite{Elliott.AF} Elliott showed that AF-algebras can be completely classified by their $K_0$-groups. These are classical \cstar-algebraic invariants which take the form of abelian preordered groups and, for AF-algebras correspond to countable \emph{dimension groups} (see \S\ref{ss.AFK0}). Elliott's result ignited what nowadays has become a prominent line of research in \cstar-algebras, usually referred to as the \emph{Elliott Classification Program}, aiming to classify larger and larger classes of separable amenable \cstar-algebras via their $K$-theoretic invariants (the bibliography on this subject is vast, as this effort spanned over decades of works by numerous researchers; see \cite{winter:survey} for an overview and further references).

In the context of model theory, AF-algebras can be interpreted as structures in the language of \cstar-algebras formulated in \cite{FHS.II} (see also \cite{Muenster}). The analysis of AF-algebras through model-theoretical lenses was approached in \cite{OmitAF}, where it is shown that the class of AF-algebras is not axiomatizable (i.e. not characterized by their first order theory, see \cite[\S 2.5]{Muenster}), yet it is expressible as an omitting type condition (at least for separable models). These results have been recast in \cite[\S 5.7]{Muenster} in terms of so-called \emph{uniform families of formulas}, in an early attempt to formalize some aspects of infinitary logic in the context of \cstar-algebras.

Of a different flavor is the work done in \cite{scowcroft}, where the author investigates how certain partial segments of the first order theory of an AF-algebra are influenced by the first order theory of its dimension group, and vice versa. In this paper we consider a problem of similar nature, aiming to understand to what extent the theory of an AF-algebra is determined by the theory of its $K_0$-group, seen as a structure in the language of ordered group with a distinguished order unit (see Remark \ref{remark:interpret}). We consider this question in the context of the infinitary logic $\mathcal{L}_{\omega_1 \omega}$. This change of perspective proves to be extremely effective, leading to what can be interpreted as a model-theoretic analogue of what Elliott proved in \cite{Elliott.AF}.

\begin{theoremi}[Corollary~\ref{cor:K0toalg}] \label{mainthm:K0toalg}
Let $A$ and $B$ be unital AF-algebras and fix $\alpha < \omega_1$. Then
\[
(K_0(A),K_0(A)_+,[1_A]) \equiv_{\omega \cdot \alpha} (K_0(B),K_0(B)_+,[1_B]) \Rightarrow A \equiv_\alpha B.
\]
\end{theoremi}

The proof of Theorem~\ref{mainthm:K0toalg} makes essential use of back-and-forth arguments, which we present in the form of Ehrenfeucht-Fra\"iss\'e games. These games are played by two players on two given structures $M$ and $N$ in the same language. In each round the players pick elements from the structures, with Player I trying to prove that $M$ and $N$ are different, and Player II claiming instead that they look the same. In the discrete setting, it is well-known that the existence of winning strategies for Player II in the \emph{Dynamic Ehrenfeucht-Fra\"iss\'e game} $\EFD_\alpha(M,N)$ is equivalent to $M \equiv_\alpha N$ (see \cite[Chapter 7]{van_games}, or Definition \ref{def:EFD} and Theorem~\ref{thm:EFD}). In an attempt to define games capable of detecting the relation $\equiv_\alpha$ in the continuous setting, we introduce a family of \emph{Partial Isomorphism games} $\PI_\alpha(A,B)$, for metric structures $A$ and $B$, coding the existence of partial isomorphisms between finitely generated substructures of $A$ and $B$ (Definition \ref{def.PI}). Our definition differs from other adaptations of EF-games to the metric setting that have already appeared in the literature (see \cite{brad, isaac} and \cite{hirvonen2021games}).

Even though it is unclear whether the existence of winning strategies for Player II in $\PI_\alpha(A,B)$ is equivalent to $A \equiv_\alpha B$, it at least implies it (Proposition~\ref{prop:PI}). In view of this, the proof of Theorem~\ref{mainthm:K0toalg} reduces to a transfer of winning strategies (for Player II) for $\EFD_{\omega \cdot \alpha}(K_0(A), K_0(B))$ to winning strategies for $\PI_\alpha(A,B)$. Concretely, this amounts to lifting partial isomorphisms between the $K_0$-groups to partial isomorphisms between the corresponding AF-algebras. This lifting process is at the heart of Elliott's classification of AF-algebras, and the setting provided by games allows us to directly use the contents of \cite{Elliott.AF} to conclude the proof of Theorem~\ref{mainthm:K0toalg}. Indeed, rephrasing the contents of \cite{Elliott.AF} in terms of games, what Elliott shows with his intertwining argument, is precisely how Player II can devise a winning strategy for the infinite version of the $\PI$ game for the algebras of interest, granted a winning strategy for the infinite version of the $\EFD$ game for the $K_0$-groups.

In the second part of the paper, we use Theorem~\ref{mainthm:K0toalg} to complement the findings of \cite[Theorem~3(3)]{OmitAF}. There it is shown that there exist unital AF-algebras which are elementary equivalent (i.e. they agree on all first order sentences) but not isomorphic. In fact, much more is true: the first order theory provides a smooth invariant (\cite[Theorem~1.1]{FTT}), while the isomorphism relation between AF-algebras is not smooth (since the relation of isomorphism between countable, torsion free, rank one abelian groups, which is not smooth, is Borel reducible to it; see \cite[Theorem~3]{OmitAF} and \cite{Thomas.TorsionFree, FTT}). Using this difference in complexity, the authors of \cite{OmitAF} conclude that the relations $\equiv$ and $\cong$ must be different on unital AF-algebras, without providing concrete examples. Albeit commutative examples are not too hard to construct (and have been implicitly treated in \cite{EV.Sat}), no simple examples, that is with no non-trivial closed two-sided ideals, were known up to this point. We employ Theorem~\ref{mainthm:K0toalg} to fill this gap and actually obtain something stronger, by building a family of simple unital AF-algebras of arbitrarily high Scott rank, as defined in \cite{SRmetric}.

\begin{theoremi}[Corollary~\ref{corollary:scott}]\label{mainthm:highrank}
There exists a family of unital simple AF-algebras $\{A_\alpha\}_{\alpha < \omega_1}$ which are pairwise non-isomorphic and such that $A_{\alpha}\equiv_\alpha A_{\beta}$ whenever $\alpha<\beta$.
 \end{theoremi}

Theorem~\ref{mainthm:K0toalg}, along with the fact that dimension groups determine their AF-algebras up to isomorphism, allows to reduce the proof of Theorem~\ref{mainthm:highrank} to finding a family of simple dimension groups $\{\mathcal{G}_\alpha \}_{\alpha < \omega_1}$ which are pairwise non-isomorphic and verify $\mathcal{G}_\alpha \equiv_\alpha \mathcal{G}_\beta$ whenever $\alpha < \beta$. This is done in \S\ref{SS.groups}. The simple dimension groups that we consider are sets of continuous functions $C(\gamma + 1, \mathbb{Q})$ with the strict order, where $\gamma$ is a countable ordinal and $\mathbb{Q}$ is considered with the discrete topology. Both $\equiv_\alpha$ and the isomorphism relation between such groups are shown to tightly depend on the countable ordinals defining the domain. This permits to further reduce the problem to a question about countable linear orders, solved using a classical result of Karp (\cite{karp}).

Although Theorem~\ref{mainthm:highrank} and its proof in \S \ref{SS.groups} can be formulated with no explicit reference to Scott rank, these notions originally motivated the construction of the family of AF-algebras in \S\ref{SS.groups} (and, to some extent, Theorem~\ref{mainthm:K0toalg} as well). 

The notions of {Scott sentence}, {Scott rank} and {Scott spectrum} are classically formulated for discrete infinitary logics, and we refer to \cite[Chapter 7]{van_games} for formally precise definitions of these concepts (the author uses the terminology \emph{Scott height}, rather than Scott rank). The \emph{Scott sentence} $\sigma_M$ of a structure $M$ is a formula coding the infinitary theory of $M$. In case $M$ is a countable $\mathcal{L}$-structure in a countable language $\mathcal{L}$, then $\sigma_M$ is a $\mathcal{L}_{\omega_1 \omega}$-sentence. In this case, the main feature of $\sigma_M$ is the following: if $N$ is another countable $\mathcal{L}$-structure satisfying $N \models \sigma_M$, then $M \cong N$. Scott's proof of the existence of a Scott sentence from \cite{scott} gives rise to the notion of \emph{Scott rank} of a structure $M$. Different definitions for the Scott rank of the structure $M$ have been considered in the literature (see e.g. \cite[\S 2.2]{marker} or \cite{scott.rank}), but in every case the Scott rank is roughly equal to the quantifier rank of $\sigma_M$ and thus measures the model-theoretic complexity of $M$.

In \cite{SRmetric} the authors generalize a large portion of Scott's analysis to the continuous framework. Suitable adaptations of Scott sentence $\sigma_M$ and Scott rank $\text{SR}(M)$ of a metric structure $M$ are introduced and it is proved that their behavior is similar to that of the classical setting.\footnote{The generalization of Scott sentence and Scott rank to the metric setting is a bit more delicate than what we make it look like in this introduction: there are some technical obstacles forcing the authors of \cite{SRmetric} to develop their notions restricting to certain classes of suitably uniform formulas. Luckily, these issues can be overcome (see \cite[\S 5]{SRmetric}), so we will ignore them in our discussion here.} According to \cite[Definition 3.6]{SRmetric}, the Scott rank of a separable structure $M$ verifies $\rk(\sigma_M) = \text{SR}(M) + \omega$, hence the existence of a separable structure $N$ not isomorphic to $M$, and such that $M \equiv_\alpha N$, immediately gives a lower bound on $\rk(\sigma_M) $, and therefore on $ \text{SR}(M)$, depending on $\alpha$. This argument can be applied to conclude that the AF-algebras produced in Theorem~\ref{mainthm:highrank} have arbitrarily high Scott rank.

Finally, we prove a partial converse to Theorem~\ref{mainthm:K0toalg}. Curiously, while in the classification of AF-algebras showing that $A \cong B$ implies $K_0(A) \cong K_0(B)$ is the easy part of the proof, the model-theoretic version of this implication seems to be more elusive. We obtain the following generalization to $\mathcal L_{\omega_1\omega}$ of \cite[Theorem~3.11.1]{Muenster}, where an analogous statement is proved for first order logic.

\begin{theoremi}[Theorem~\ref{thm:algbstoK_0}] \label{mainthm:reverse}
Let $A$ and $B$ be unital \cstar-algebras, let $\alpha<\omega_1$ and let $\mathcal{K}$ be the \cstar-algebra of the compact operators on a separable infinite-dimensional Hilbert space. Then
\[
A\otimes\mathcal K\equiv_{\omega + 2\cdot \alpha +2} B\otimes\mathcal K \Rightarrow (K_0(A),K_0(A)_+,[1_A]) \equiv_{ \alpha} (K_0(B),K_0(B)_+,[1_B]).
\]
\end{theoremi}

We ultimately obtain the following correspondence between elementary equivalence of AF-algebras (metric structures) and elementary equivalence of unital ordered groups (discrete structures).
In the following corollary, we suggestively write $\equiv_{ \alpha}^{\PI}$ for the statement that Player~II has a winning strategy for the game $\PI_\alpha(A,B).$
\begin{corollaryi}[Corollary~\ref{thm:combined}] \label{mainthm:combined}
	Let $A$ and $B$ be unital AF-algebras, let $\alpha<\omega_1$ be a non-zero ordinal which is such that $\alpha = \omega^\omega \cdot \beta,$ for some $\beta  \leq \alpha.$  Then the following are equivalent
	
	\begin{enumerate}
		\item\label{el.eq1D} $A \equiv_{ \alpha}^{\PI} B$;
		\item\label{el.eq2D} $A\otimes\mathcal K\equiv_\alpha B\otimes\mathcal K$;
		\item\label{el.eq3D} $(K_0(A),K_0(A)_+,[1_A]) \equiv_{ \alpha} (K_0(B),K_0(B)_+,[1_B]).$
	\end{enumerate}
\end{corollaryi}
Despite relating $\equiv_{ \alpha}^{\PI}$ to $\equiv_\alpha$, Corollary \ref{mainthm:combined} gives no additional
information on the possibility of characterizing $A \equiv_\alpha B$ in terms of winning strategies for $\PI_\alpha(A,B)$, even
when restricting to the class of AF-algebras. This gap is summarized in the following question.


\begin{question}\label{que:eleq}
	Let $\alpha < \omega_1$ and let $A$ and $B$ be metric structures (e.g. \cstar-algebras) with $A \equiv_{ \alpha} B$. Under what conditions on $\alpha, A, B$ can one infer that Player II has a winning strategy for 
the game $\PI_\alpha(A,B)$? More generally, can $\equiv_\alpha$ be expressed in terms of existence of winning strategies
for games (possibly different from our $\PI_\alpha$) also in the metric setting?
\end{question}

Regardless of the answer to Question \ref{que:eleq},
Corollary~\ref{thm:combined} suggests that $\equiv_{ \alpha}^{\PI}$ might be a natural and interesting notion of infinitary elementary equivalence between \cstar-algebras (or at least AF-algebras).

As mentioned at the beginning of this introduction, the classification of AF-algebras by $K_0$-groups was the first step in a long series of results which eventually lead to the classification of a wide class of suitably amenable \cstar-algebras. Since Elliott's proof in \cite{Elliott.AF}, the methods and the techniques employed in such classification have been greatly refined and improved, and the invariant itself has evolved to a much more sophisticated object (known as the \emph{Elliott Invariant}), of which the $K_0$-group is just a part. Nevertheless, the core strategy for most results in the Elliott Classification Program has remained the same. These proofs are often composed of an \emph{existence theorem}, claiming that every morphism between the invariants can be lifted to a $^*$-homomorphism between the corresponding algebras (see e.g. Lemma~\ref{lemma:classification}\ref{item:ex}), a \emph{uniqueness theorem}, granting that $^*$-homomorphisms that induce the same maps on the invariants are approximately unitarily equivalent (see e.g. Lemma~\ref{lemma:classification}\ref{item:uni}), and finally an \emph{intertwining argument}, putting together the two previous steps to obtain the desired classification result. This template combines surprisingly well with the model-theoretic framework provided by games, as witnessed by the proof of Theorem~\ref{mainthm:K0toalg} in \S\ref{SS.Strategies}. It seems thus natural to ask for what other classes of classifiable \cstar-algebras this approach can be employed, with the final goal in mind of understanding how well the Elliott invariant of a classifiable \cstar-algebra captures the model-theoretic features of the algebra, and vice versa.

A first step to verify this intuition would be considering Kirchberg-Phillips classification of purely infinite \cstar-algebras (\cite{kirch,phillips,Gabe}), and investigate whether those results could fuel an analogue of Theorem~\ref{mainthm:K0toalg} for purely infinite classifiable \cstar-algebras.

\begin{question}
Fix $\alpha < \omega_1$. Does there exist $\theta(\alpha)$ such that for every unital, purely infinite, simple, nuclear, separable \cstar-algebras $A$ and $B$ satisfying the UCT we have that
\[
K_0(A) \equiv_{\theta(\alpha)} K_0(B) \text{ and } K_1(A) \equiv_{\theta(\alpha)} K_1(B)
\Rightarrow A \equiv_\alpha B?
\]
\end{question}

The paper is organized as follows. \S\ref{S.prelim} is dedicated to preliminaries; most importantly, we introduce the Partial Isomorphism game (Definition \ref{def.PI}). In \S\ref{SS.Strategies} we prove Theorem~\ref{mainthm:K0toalg}, while \S\ref{SS.groups} is devoted to the construction of the family of simple AF-algebras promised in Theorem~\ref{mainthm:highrank}. In \S\ref{S.AlgtoK0} we prove Theorem \ref{mainthm:reverse} and Corollary \ref{mainthm:combined}, and we obtain a transfer of strategies from the Partial Isomorphism game played with a pair of given \cstar-algebras to the game played with their stabilizations. Finally, in the appendix we give a concrete description of the Bratteli diagrams associated to the AF-algebras built in \S\ref{SS.groups}.

\section{Preliminaries}\label{S.prelim}
We assume that the reader is familiar with the basics of the theory of \cstar-algebras (see \cite{ilijas_book, murphy} for an essential introduction on the topic), as well as the basics of model theory for metric structures and its formulation in the setting of \cstar-algebras (standard references are \cite{BYBHU} and \cite{FHS.II}; see also \cite{Muenster}).

We briefly recall some elementary notions in infinitary logic, both in the discrete and the continuous setting (\S\ref{ss.mtd} and \S\ref{ss.mtc}), we define the \emph{Partial Isomorphism Game} (Definition \ref{def.PI}), and we provide some essential background on AF-algebras and their dimension groups (\S\ref{ss.AFK0}).

\subsection{Model theory: the Discrete Setting} \label{ss.mtd}
\subsubsection{Dynamic Ehrenfeucht-Fra\"iss\'e Games}
We refer to \cite{van_games} for a thorough and rigorous introduction to games in model theory. In this paper we are mainly concerned with the dynamic version of Ehrenfeucht-Fra\"iss\'e games, which we simply call EFD-games (see \cite[\S 7.3]{van_games}), and their adaptation to the metric setting, which we introduce in Definition \ref{def.PI}. 

If $\mathcal L$ is a language for discrete structures, $A$ is an $\mathcal L$-structure and $(a_0,\ldots,a_{k-1})\in A^k$, then $\langle a_0,\ldots,a_{k-1}\rangle_\mathcal{L}$ is the $\mathcal L$-structure generated by the $a_i$'s.

\begin{definition}[The Dynamic EF Game] \label{def:EFD}
Let $\mathcal L$ be a language and $A$ and $B$ two $\mathcal L$-structures. Let $\alpha$ be an ordinal. The game $\EFD_\alpha(A,B)$ is played as follows. Set $\alpha_{-1}=\alpha$. In round $k\in\N$, the players first check if $\alpha_{k-1} > 0$. If this is the case, then
\begin{enumerate}
\item Player I either plays a pair $(\alpha_k, a_k)$ or a pair $(\alpha_k, b_k)$, where $\alpha_k < \alpha_{k-1}$, and $a_k \in A$ or $b_k\in B$, respectively.
\item Player II answers with an element $b_k \in B$ if Player I played $a_k \in A$, while they answer with some $a_k \in A$ if Player I played $b_k \in B$.
\end{enumerate}
If $\alpha_{k-1} = 0$ the game ends, and Player II wins if the map $a_i\mapsto b_i$ induces an isomorphism between $\langle a_0,\ldots,a_{k-1}\rangle_\mathcal{L}$ and $\langle b_0,\ldots,b_{k-1}\rangle_\mathcal{L}$; otherwise Player I wins.
\end{definition}

Fix an ordinal $\alpha$ and two $\mathcal L$-structures $A$ and $B$. Given $k \in \bbN$,
a \emph{$k$-position} for $\EFD_\alpha(A,B)$ is a sequence
\[
\bar p = ((\alpha_0, a_0, b_0), \dots, (\alpha_{k-1}, a_{k-1}, b_{k-1}))
\]
obtained from a match of $\EFD_\alpha(A,B)$, and representing the status of the game at the beginning of round $k$. In particular, $\bar p$ is such that $\alpha > \alpha_0 > \dots > \alpha_{k-1}$, that $a_j \in A$ and that $b_j \in B$ for all $j < k$. In case $k=0$ we set $\bar p$ to be $(\alpha, \emptyset, \emptyset)$, the position that simply records the ordinal $\alpha$, which is sometimes referred to as the \emph{empty position}. Let $\mathcal{P}_k(\EFD_\alpha(A,B))$ be the set of all $k$-positions for $\EFD_\alpha(A,B)$. A \emph{strategy} $\bar \sigma$ for Player II is a sequence of maps
\[
\sigma_k \colon\mathcal{P}_k(\EFD_\alpha(A,B)) \times\alpha \times A \sqcup B \to A \sqcup B, \ \forall k \in \N,
\]
such that $\sigma_k(\bar p, \beta, c)$ belongs to $A$ (respectively, to $B$) whenever $\bar p = (\alpha_j, a_j, b_j)_{j < k}\in \mathcal{P}_k(\EFD_\alpha(A,B))$ with $\beta < \alpha_{k-1}$ and $c \in B$ (respectively, $c \in A$). Player II has a \emph{winning strategy} for $\EFD_\alpha(A,B)$, in symbols $\text{II} \uparrow \EFD_\alpha(A,B)$, if there is a strategy $\bar \sigma = (\sigma_k)_{k\in \N}$ such that Player II wins every match of $\EFD_\alpha(A,B)$ where they play their moves according to $\bar \sigma$. That is, if $\bar p = (\alpha_j, a_j, b_j)_{j < k} \in \mathcal{P}_k(\EFD_\alpha(A,B))$ is such that
\begin{enumerate}
\item $\alpha_{k-1} = 0$,
\item $b_j = \sigma_j((\alpha_i, a_i, b_i)_{i < j}, (\alpha_j, a_j))$, for all $j < k$ such that Player I played $(\alpha_j, a_j)$ in round $j$, 
\item $a_j = \sigma_j((\alpha_i, a_i, b_i)_{i < j}, (\alpha_j, b_j))$ for all $j < k$ such that Player I played $(\alpha_j, b_j)$ in round $j$,
\end{enumerate}
then Player II wins the match that ends in position $\bar p$. Analogously, given a $k$-position $\bar p$, we say that Player II has a \emph{winning strategy from position $\bar p$} if there exists a strategy $\bar \sigma$ such that Player II wins every match of the game $\EFD_\alpha(A,B)$ which at the beginning of round $k$ is in position $\bar p$, and where Player II starts playing according to $\bar \sigma$ from that point on.

The notions of strategies and winning strategies can be also defined for Player I, but we will be mostly interested in games where Player II can win.

The ordinal $\alpha$ in Definition \ref{def:EFD} plays the role of a clock, which is refreshed every turn by Player I, and determines the end of the game when it reaches zero. Note that $\EFD_\alpha$ always ends in a finite number of moves, even if the ordinal $\alpha$ is infinite.

\subsubsection{Discrete Infinitary Logic} \label{ss.dil}
Given a language $\mathcal{L}$, the set of $\mathcal{L}_{\omega_1 \omega}$-formulas expands first order logic by allowing, given a countable set $\{\phi_i(\bar x)\}_{i \in \N}$ of formulas in a finite set of variables $\bar x$, countable conjunctions $\bigwedge_i \phi_i(\bar x)$ and disjunctions $\bigvee_i \phi_i(\bar x)$ of formulas. 

The quantifier rank of $\mathcal L_{\omega_1\omega}$-formulas is defined by induction as follows
\begin{enumerate}
\item $\rk(\phi(\bar x))=0$ if $\phi(\bar x)$ is atomic;
\item $\rk(\neg \phi(\bar x))=\rk(\phi(\bar x))$;
\item $\rk (\bigwedge_i\phi_i(\bar x))=\rk(\bigvee_i\phi_i(\bar x))=\sup_i\rk(\phi_i(\bar x))$;
\item $\rk(\exists y\phi(\bar x,y))=\rk(\forall y\phi(\bar x,y))=\rk(\phi(\bar x,y))+1$.
\end{enumerate}
The quantifier rank of an $\mathcal{L}_{\omega_1 \omega}$-formula is always a countable ordinal. Given a language $\mathcal{L}$, the relation $\equiv_\alpha$ between $\mathcal{L}$-structures, for an ordinal $\alpha$, is defined as\footnote{The definition of $\equiv_\alpha$ customarily requires that the two structures agree on all $\mathcal{L}_{\infty \omega}$-sentences of rank at most $\alpha$. This is not a significant difference in our context, since in this paper we will only consider countable structures.} 
\[
A\equiv_\alpha B\text{ if and only if $A$ and $B$ agree on all $\mathcal{L}_{\omega_1\omega}$-sentences $\phi$ with $\rk(\phi) \le \alpha$.}
\]
The equivalence relations $\equiv_\alpha$ are tightly related to the $\EFD$-games:
\begin{theorem}[{\cite[Theorem~7.47]{van_games}}]\label{thm:EFD}
Let $\mathcal L$ be a countable language and let $\alpha < \omega_1$. Let $A$ and $B$ be countable $\mathcal L$-structures. Then $A\equiv_\alpha B$ if and only if $\text{II} \uparrow\EFD_\alpha(A,B)$.
\end{theorem}

\subsection{Model Theory: the Continuous Setting} \label{ss.mtc}
\subsubsection{Metric Infinitary Logic} \label{sss.metricinf}
The majority of the metric structures considered in this paper are unital \cstar-algebras (with the exception of \S \ref{S.AlgtoK0}). In the context of continuous model theory, unital \cstar-algebras are customarily defined as one-sorted structures with multiple domains of quantification $D_n$, each one interpreted as the norm $n$-ball centered at 0, in the language $\mathcal{L}_{\mathrm{C}^\ast}$, containing symbols for each algebraic operation $+$, $\cdot$ and $^*$, constant symbols 0 and 1, and unary function symbols $\lambda$ for every $\lambda \in \mathbb{C}$. The distance on a \cstar-algebra is determined by the \cstar-algebraic norm $\lVert \cdot \rVert$. We refer to \cite[\S 2.3.1]{FHS.II} and \cite[\S 2.1]{Muenster} for a rigorous and complete introduction to this setup. All unital \cstar-algebras appearing in this paper are implicitly considered as $\mathcal{L}_{\rm C^\ast}$-structures.

Some of the definitions and preliminary results of this section also apply outside the realm of \cstar-algebras,
so we state them in the broader framework of general metric structures from \cite{BYBHU}.
Terms and first order formulas are defined by recursion on the complexity as in \cite[\S 3]{BYBHU}, using continuous
functions as connectives and $\inf$ and $\sup$ as quantifiers. Given a formula $\phi(x_0, \dots, x_{n-1})$ and a metric structure $M$ in the same language, we let $\phi^M \colon M^n \to \mathbb{R}$ denote the evaluation of $\phi$ on $M$, as defined in \cite[Definition 3.3]{BYBHU}. Each formula $\phi(\bar x)$ comes equipped with a modulus of continuity (see \cite[Definition 2.1]{SRmetric}) $\Delta_\phi\colon \mathbb{R}^n \to \mathbb{R}^+$ and an interval $I_\phi \subseteq \mathbb{R}$. For every metric structure $M$ with a distance $d$ the map $\phi^M$ is a uniformly continuous function such that
\[
\lvert \phi^M(\bar a) - \phi^M(\bar b) \rvert < \Delta_\phi(d(a_0,b_0), \dots, d(a_{n-1}, b_{n-1})), \ \forall \bar a, \bar b \in M^n,
\]
and such that the range of $\phi^M$ is contained in $I_\phi$.

When defining $\mathcal{L}_{\omega_1 \omega}$-formulas in continuous logic, we want to make sure that their interpretation is still bounded and uniformly continuous. For this reason, following the setup developed in \cite{BYI.MTForcing, SRmetric}, we only allow for infinite conjunctions and disjunctions of sufficiently uniform sets of formulas. In particular, suppose that $\{ \phi_i(\bar x) \}_{i \in \N}$ is a family of formulas in a finite set of variables $\bar x$, for which there is a modulus of continuity $\Delta$ and an interval $I \subset \mathbb{R}$ such that $\Delta_{\phi_i} \le \Delta$ and $I_{\phi_i} \subseteq I$ for every $i \in \N$. Then
\[
\phi(\bar x) = \bigwedge_i \phi_i(\bar x), \ \psi(\bar x) = \bigvee_i \phi_i(\bar x),
\]
are both formulas and their interpretation for a tuple $\bar a \in M^n$ is given by
\[
\phi^M(\bar a) := \inf_i \phi_i^M(\bar a) , \ \psi^M(\bar a) := \sup_i \phi_i^M(\bar a).
\]
This seemingly innocuous restriction leads to various non-trivial obstacles of technical nature when trying to adapt some of the classical arguments in infinitary logic to the continuous setting (the contents of \cite{SRmetric} provide a good example of this).

\begin{remark}
The setup of \cite{BYBHU} and \cite{SRmetric}, which is our main reference for infinitary logic of metric structures, formalizes continuous model theory for multi-sorted bounded structures. On the other hand, in the literature \cstar-algebras are usually presented as one-sorted structures, and their unboundedness is dealt with by introducing domains of quantification $D_n$, each one representing the ball of all elements of norm at most $n$. This difference is mainly formal, and there are a number of remedies for this discrepancy. For instance, the structure of a \cstar-algebra can be completely recovered from its unit ball (where, instead of the operation $+$ we would need to consider $\frac{x+y} {2}$), an object that perfectly fits the framework of \cite{BYBHU, SRmetric}. We will thus apply various definitions and results from \cite{SRmetric} to $\mathcal{L}_{\rm C^\ast}$-structures with this proviso in mind, while still treating \cstar-algebras following the setup from \cite{FHS.II, Muenster}, which has the main advantage of allowing quantification and terms outside the unit ball. To give a concrete example, given a family of $\mathcal{L}_{\rm C^\ast}$-formulas $\{ \phi_i(\bar x) \}_{i \in \N}$, the formulas $\bigwedge_i \phi_i(\bar x)$ and $\bigvee_i \phi_i(\bar x)$ are well-defined if and only if there is a modulus of continuity $\Delta$ and an interval $I \subset \mathbb{R}$ such that, for every $ i \in \bbN$, we have $\Delta^1_{\phi_i} \le \Delta$ and $I^1_{\phi_i} \subseteq I$, where $\Delta^1_{\phi_i}$ and $I^1_{\phi_i}$ are the modulus of continuity and interval associated to $\phi_i$ on the domain $D_1$.
\end{remark}

Similarly to the discrete setting, the rank of an $\mathcal L_{\omega_1\omega}$-formula $\phi(\bar x)$ is defined by induction as follows
\begin{enumerate}
\item $\rk(\phi(\bar x))=0$ if $\phi(\bar x)$ is atomic;
\item if $f\colon\mathbb R^n\to\mathbb R$ is a uniformly continuous function, then $\rk(f(\phi_0(\bar x),\ldots,\phi_{n-1}(\bar x)))=\max_i\rk(\phi_i)(\bar x)$;
\item\label{clause3} $\rk (\bigwedge_i\phi_i(\bar x))=\rk(\bigvee_i\phi_i(\bar x))=\sup_i\rk(\phi_i(\bar x))$;\item $\rk(\inf_y\phi(\bar x,y))=\rk(\sup_y\phi(\bar x,y))=\rk(\phi(\bar x,y))+1$.
\end{enumerate}

If $A$ and $B$ are metric structures in the same language, let 
\[
A\equiv_\alpha B\text{ if and only if }\phi^A=\phi^B\text{ for all $\mathcal{L}_{\omega_1 \omega}$-sentences $\phi$ with } \rk(\phi) \le \alpha.
\]

\subsubsection{The Partial Isomorphism Game}
In this section we introduce an analogue of the EFD-game in Definition \ref{def:EFD} which is suitable for metric structures.

If $\mathcal L$ is a language for metric structures, $A$ is an $\mathcal L$-structure and $(a_0,\ldots,a_{k-1})\in A^k$, we denote by $\langle a_0,\ldots,a_{k-1}\rangle_\mathcal{L}$ the $\mathcal L$-structure generated by the $a_i$'s. In case $\mathcal{L} = \mathcal{L}_{\rm C^\ast}$ and $A$ is a unital \cstar-algebra, $\langle a_0,\ldots,a_{k-1}\rangle_{\mathcal{L}_{\rm C^\ast}}$ is the \cstar-algebra generated by $1_A, a_0, \dots, a_{k-1}$.

\begin{definition}[The Partial Isomorphism Game] \label{def.PI}
Let $\mathcal L$ be a language and $A$ and $B$ two metric $\mathcal L$-structures, with distances $d_A$ and $d_B$ respectively. Let $\alpha$ be an ordinal. The game $\PI_\alpha(A,B)$ is played as follows. Set $\alpha_{-1} = \alpha$. In round $k \in \N$ the players first check if $\alpha_{k-1} > 0$. If this is the case, then 
\begin{enumerate}
\item Player I plays a triple $(\alpha_k, c_k,\e_k)$ where $\alpha_k < \alpha_{k-1}$, $\e_k > 0$, and $c_k \in A \sqcup B$.
\item Player II answers with a pair $( a_k, b_k) \in A \times B$ such that $d_A( a_k,  c_k) < \e_k$ if $c_k \in A$,
or such that $d_B( b_k , c_k ) < \e_k$ if $c_k \in B$.
\end{enumerate}
If $\alpha_{k-1} = 0$ the game ends, and Player II wins if the map $ a_i\mapsto b_i$ induces an isomorphism between $\langle a_0,\ldots, a_{k-1}\rangle_\mathcal{L}$ and $\langle b_0,\ldots, b_{k-1}\rangle_\mathcal{L}$; otherwise Player I wins.
\end{definition}

The terminology and notation introduced for EFD-games (e.g. position, strategy, winning strategy, $\text{II} \uparrow\PI_\alpha(A,B)$, etc.) can be generalized, mutatis mutandis, to PI-games.

The main difference between EFD-games and PI-games is the presence of the $\e_k$'s, which allow to perturb Player I's moves and make it easier for Player II to have winning strategies. Such flexibility is something to be expected in the metric setting, as otherwise Player II can rarely hope to win. To see this, consider for instance AF-algebras, which we introduce in the next section. All such \cstar-algebras are singly generated (\cite{thiel}), so if we start a match of $\PI_\alpha(A,B)$ where $A$ and $B$ are two such objects, then Player I can play both generators in the first two rounds. If Player II is not allowed to perturb Player~I's move, then they have a winning strategy if and only if $A \cong B$. 

Even with the extra flexibility granted by $\e_k$, it is not at all clear whether one can hope for a characterization of $\equiv_\alpha$ in terms of PI-games like the one we have in the discrete setting with Theorem~\ref{thm:EFD} (see Question~\ref{que:eleq}). It is not even clear if the conjunction of $\text{II} \uparrow\PI_\alpha(A,B)$ and $\text{II} \uparrow\PI_\alpha(B,C)$ implies $\text{II} \uparrow\PI_\alpha(A,C)$, that is whether the relation induced by PI-games on structures is transitive, although for AF-algebras this follows, for all $\alpha$ of the form $\omega^\omega \cdot \beta$, from Corollary~\ref{thm:combined}. Nevertheless, the following partial analogue of Theorem~\ref{thm:EFD} is all we need later in the paper.

\begin{proposition} \label{prop:PI}
Let $\mathcal L$ be a language for metric structures and let $\alpha<\omega_1$. Let $A$ and $B$ be $\mathcal L$-structures. Then $\text{II} \uparrow\PI_\alpha(A,B)$ implies $A \equiv_\alpha B$.
\end{proposition}
\begin{proof}
The proposition is a consequence of the following claim applied to the empty position.
\begin{claim}
Let $\bar p \in \mathcal{P}_k(\PI_\alpha(A,B))$ be the position $\bar p = ((\alpha_j, c_j,\e_j, a_j, b_j))_{j < k}$. If Player~II has a winning strategy from $\bar p$, then $(A, a_0, \dots, a_{k-1}) \equiv_{\alpha_{k-1}} (B, b_0, \dots, b_{k-1})$.
\end{claim}
The proof is by induction on $\alpha_{k-1}$. Suppose that $\alpha_{k-1} = 0$ and let $\phi(\bar x)$ be a formula with $\rk(\phi(\bar x)) = 0$. The value of $\phi(\bar x)$ does not change when going to substructures, as its rank is zero. By assumption we have that $C :=\langle a_0,\ldots, a_{k-1}\rangle_\mathcal{L}\cong\langle b_0,\ldots, b_{k-1}\rangle_\mathcal{L} =: D$, thus we have
\[
\phi^A( a_0,\ldots, a_{k-1}) = \phi^{C}( a_0,\ldots, a_{k-1}) = 
\phi^{D}( b_0,\ldots, b_{k-1}) = \phi^B( b_0,\ldots, b_{k-1}).
\]
It follows that
\[
(A, a_0,\ldots, a_{k-1}) \equiv_{0} (B, b_0,\ldots, b_{k-1}).
\]
Suppose now that $\alpha_{k-1}$ is a limit ordinal and that the claim has been proved for all positions $\bar q = ((\beta_j, c_j', \e_j', a_j', b_j'))_{j < k'}$ with $\beta_{k-1} < \alpha_{k-1}$. When applied to empty positions this entails that $A \equiv_\beta B$ for all $\beta < \alpha_{k-1}$, which in turn gives $A \equiv_{\alpha_{k-1}} B$, and thus proves the claim if $\bar p$ is the empty position. Suppose then that $\bar p$ is a $k$-position for $k > 0$ and set
\[
\bar p'= ((\alpha_j, c_j,\e_j, a_j, b_j))_{j < k-1}
\]
Pick $\gamma < \alpha_{k-1}$. By assumption Player II has a winning strategy $\bar \sigma$ from $\bar p$. It is immediate to see that $\bar \sigma$ is also a winning strategy from the position $\bar p'^\smallfrown(\gamma,c_{k-1},\e_{k-1}, a_{k-1}, b_{k-1})$. By inductive hypothesis we conclude that $(A, a_0, \dots, a_{k-1}) \equiv_{\gamma} (B, b_0, \dots, b_{k-1})$ for every $\gamma < \alpha_{k-1}$, which implies $(A, a_0, \dots, a_{k-1}) \equiv_{\alpha_{k-1}} (B, b_0, \dots, b_{k-1})$.

Suppose finally that $\alpha_{k-1} = \beta + 1$ and that the claim holds for positions ending with the clock on value $\beta$. Let
\[
\phi(\bar x) = \sup_y \psi(\bar x, y),
\]
 where $\psi(\bar x, y)$ is some formula of rank $\beta$. Let $r=\phi^A( a_0,\dots, a_{k-1})$ and $s=\phi^B( b_0,\dots, b_{k-1})$. Suppose by contradiction that there is $\e>0$ such that $s +\e<r$. By uniform continuity, we can choose $\delta>0$ such that, denoting $d_A$ the distance defining the metric structure on $A$,
 \[
d_A(a,a') <\delta \Rightarrow \lvert\psi^A( a_0,\ldots, a_{k-1},a)-\psi^A( a_0,\ldots, a_{k-1},a')\rvert<\e/2,\ \forall a,a' \in A.
 \] 
Let $c_{k} \in A$ be such that
\[
| \psi^A( a_0, \dots, a_{k-1}, c_{k}) -r| < \e/2,
\]
and suppose that Player I plays $(\beta, c_{k},\delta)$ in round $k$ of $\PI_\alpha(A,B)$. Since Player II has a winning strategy $\bar \sigma$ from position $\bar p$, there exists a move $( a_{k}, b_{k})$ (namely the one indicated by $\bar \sigma$) such that $\bar \sigma$ is also a winning strategy from position $\bar p^\smallfrown(\beta,c_{k},\delta,a_{k}, b_{k})$. By inductive assumption $(A, a_0, \dots, a_{k}) \equiv_{\beta} (B, b_0, \dots, b_{k})$, hence $\psi^A( a_0, \dots, a_{k}) = \psi^B( b_0, \dots, b_{k})$. Since $d_A(c_{k}, a_{k}) <\delta$, we have that $\lvert r-\psi^A( a_0, \dots, a_{k})\rvert<\e$, hence 
\[
|\psi^B( b_0, \dots, b_{k})-r|<\e,
\]
and therefore $|s-r|<\e$, a contradiction. (The case in which $r+\e<s$ is treated in the same way.)

The case $\phi(\bar x) = \inf_y \psi(\bar x, y)$ is analogous, and the case when $\phi(\bar x)$ is a general formula of rank $\beta +1$ can be proved by induction on the complexity.
\end{proof}

\subsection{AF-algebras and Dimension Groups} \label{ss.AFK0}

\subsubsection{AF-algebras and the $K_0$-functor}
In this section we recall some basic facts on AF-algebras and their $K_0$-groups (we refer to \cite{rll_ktheory} for a complete overview of these topics). We emphasize again that we require AF-algebras to be separable. A separable \cstar-algebra $A$ is \emph{approximately finite}, or simply an \emph{AF-algebra}, if there exists an increasing sequence $\{ A_n \}_{n \in \N}$ of finite-dimensional subalgebras of $A$ such that $\overline{\bigcup_{n \in \N} A_n} = A$. Recall that all finite-dimensional \cstar-algebras are isomorphic to direct sums of matrix algebras over the complex numbers.

We briefly recall the definition of $K_0$-group of a unital \cstar-algebra. Let $\mathcal K$ be the algebra of compact operators on a complex separable infinite-dimensional Hilbert space. If $A$ is a \cstar-algebra and $p$ and $q$ are projections in $A$, we say that $p$ and $q$ are Murray von Neumann equivalent, written $p\sim q$, if there is $v\in A$ with $vv^*=p$ and $v^*v=q$. The set of Murray von Neumann equivalence relations of projections in $A\otimes\mathcal K$ is denoted by $V(A)$. After fixing an isomorphism $\mathcal{K} \cong M_2(\mathcal{K})$, the set $V(A)$ can be naturally endowed with a semigroup structure with the operation 
\[
[p]+[q]= \left[
\begin{pmatrix}
p & 0 \\
0 & q
\end{pmatrix} \right].
\]
Given a projection $p \in A$, we let $[p] \in V(A)$ denote the class corresponding to the projection $p \otimes q \in A \otimes \mathcal{K}$, where $q$ is some minimal non-zero projection in $\mathcal{K}$. As all minimal projections in $\mathcal K$ are Murray von Neumann equivalent, this class does not depend on the choice of $q$.

The \emph{$K_0$-group} of a unital \cstar-algebra $A$, denoted $K_0(A)$, is the Grothendieck group of $V(A)$, that is the set of pairs $([p], [q]) \in V(A)^2$ with sum $([p_0], [q_0]) + ([p_1], [q_1]) = ([p_0] + [p_1], [q_0] + [q_1])$, and where the pairs $([p_0], [q_0])$ and $([p_1], [q_1])$ are identified if and only if there is $[r] \in V(A)$ such that $[p_0] + [q_1] + [r] = [p_1] + [q_0] + [r]$.

Let $K_0(A)_+ = \{ ([p] , [0]) \mid [p] \in V(A) \}$. If $V(A)$ has the cancellation property, then the map $[p] \mapsto ([p] , [0])$ is a bijection between $V(A)$ and $K_0(A)_+$, and we can identify the two sets. This is the case when $A$ is a unital AF-algebra (and more generally when $A$ is stably finite), and in this scenario $(K_0(A), K_0(A)_+,[1_A])$ is a unital ordered abelian group in the sense of the following definition (\cite[Propositions 5.1.5 and 5.1.7]{rll_ktheory}).

\begin{definition} \label{def:gg+}
A pair $(G,G_+)$ is called an \emph{ordered abelian group} if $G$ is an abelian group and $G_+$ is a positive cone in $G$ containing $0,$ i.e.\
\begin{itemize}
\item $G_+ + G_+ \subseteq G_+,$
\item $G_+-G_+=G$,
\item $G_+\cap(-G_+)=\{0\}$.
\end{itemize}
Given an ordered abelian group $(G,G_+)$ and $g$ and $h$ in $G$, we write $g\leq h$ if $h-g\in G_+$.

If $u\in G_+$, we say that $u$ is an \emph{order unit} if for all $x\in G$ there is $n$ such that $-nu\leq x\leq nu$. In this case the triple $(G,G_+,u)$ is called a unital ordered abelian group. A \emph{positive homomorphism} between two ordered groups $(G, G_+)$ and $(H, H_+)$ is a group homomorphism $\Phi \colon G \to H$ which preserves the order (or, equivalently, that maps $G_+$ into $H_+$). A positive homomorphism $\Phi$ between two unital ordered groups $(G, G_+, u)$ and $(H, H_+, v)$ is \emph{unital} if $\Phi(u) = v$.
\end{definition}

The map sending $A \mapsto (K_0(A), K_0(A)_+, [1_A])$ is a functor from the category of unital (approximately finite) \cstar-algebras into the category of (unital ordered) abelian groups. In fact, every $^*$-homomorphism $\Phi \colon A \to B$ naturally extends to a $^*$-homomorphism $ \Phi' \colon A \otimes \mathcal{K} \to B \otimes \mathcal{K}$ which in turn induces a group homomorphism $K_0(\Phi)\colon K_0(A)\to K_0(B)$ by mapping $[p] \mapsto [ \Phi'(p)]$ for every projection $p \in A \otimes \mathcal{K}$. Since $K_0(\Phi)(K_0(A)_+)$ is always contained in $K_0(B)_+$, if both $(K_0(A), K_0(A)_+)$ and $(K_0(B), K_0(B)_+)$ are ordered groups in the sense of Definition \ref{def:gg+}, then the map $K_0(\Phi)$ is moreover a positive homomorphism. If finally $\Phi$ is unital, then $K_0(\Phi)([1_A])=[1_B]$. This functor is continuous with respect to inductive limits, that is, if $A=\lim A_i$, with maps $\Phi_i \colon A_i \to A$, then $K_0(A) = \bigcup_{i \in \mathcal{I}} K_0(\Phi_i(A_i))$ and $K_0(A)_+ = \bigcup_{i \in \mathcal{I}} K_0(\Phi_i(A_i))_+$ (see \cite[Theorem~6.3.2]{rll_ktheory}).

Elliott's celebrated classification theorem of AF-algebra (\cite{Elliott.AF}) shows that the $K_0$-functor classifies AF-algebras up to isomorphism (also in the non-unital case). Its proof is reproduced in any standard text; we recommend \cite[Theorem~7.3.4]{rll_ktheory}. Elliott's proof does not simply show that $A \cong B$ if and only if $(K_0(A),K_0(A)_+,[1_A]) \cong (K_0(B),K_0(B)_+,[1_B])$, for any two AF-algebras $A$ and $B$. What it actually demonstrates is that any positive unital isomorphism $\Phi \colon K_0(A) \to K_0(B)$ lifts to an isomorphism $\Psi\colon A \to B$ such that $K_0(\Psi) = \Phi$. Two technical key tools employed in that proof are isolated in the following lemma for future reference.

\begin{lemma}[{\cite[Lemma~7.3.2]{rll_ktheory}}] \label{lemma:classification}
Let $F$ be a finite-dimensional \cstar-algebra and let $A$ be a unital AF-algebra.
\begin{enumerate}[label=(\roman*)]
\item \label{item:ex} For every positive unital homomorphism $\Phi \colon K_0(F) \to K_0(A)$ there exists a unital $^*$-homomorphism $\Psi \colon F \to A$ with $K_0(\Psi) = \Phi$.
\item \label{item:uni} Let $\Psi_0$ and $\Psi_1$ be $^*$-homomorphisms $F \to A$. Then $K_0(\Psi_0) = K_0(\Psi_1)$ if and only if there is a unitary $u \in A$ such that $\Psi_0 = \text{Ad}(u) \circ \Psi_1$.
\end{enumerate}
\end{lemma}

\subsubsection*{Dimension Groups}
In this section we quickly introduce the class of ordered abelian groups that arise as $K_0$-groups of AF-algebras. This class can be abstractly characterized by properties \eqref{item:unp}-\eqref{item:rip} of the following definition.
\begin{definition}\label{defin:properties}
Let $(G,G_+)$ be an ordered abelian group. The group $(G, G_+)$
\begin{enumerate}
\item \label{item:unp} is \emph{unperforated} if $nx\geq 0$ for some $n>0$ implies $x\geq 0$, for every $x \in G$;
\item \label{item:rip} has \emph{Riesz interpolation property} if for every $x_0,x_1,y_0,y_1\in G$ with $x_i\leq y_j$ for $i,j\in\{0,1\}$ there is $z\in G$ with $x_i\leq z\leq y_j$ for $i, j\in\{0,1\}$;
\item is \emph{simple} if every non-zero element of $G_+$ is an order unit.
\end{enumerate}
\end{definition}

\begin{theorem}[{\cite[Theorem~2.2]{EHS.Dimension}}] \label{thm:dimgr}
A (unital) countable ordered abelian group $(G, G_+)$ arises as the $K_0$-group of a (unital) AF-algebra $A$ if and only if it is unperforated and has the Riesz interpolation property. Moreover $(G, G_+)$ is simple if and only if $A$ is simple.
\end{theorem}
We say that a countable ordered abelian group is a \emph{dimension group} if it satisfies the two equivalent conditions of Theorem~\ref{thm:dimgr}. The most elementary examples of dimension groups are the $K_0$-groups of finite-dimensional \cstar-algebras. Given $F \cong \bigoplus_{\ell < m} M_{n_\ell}(\mathbb{C})$, it is an exercise to check that the resulting $K_0$-group is the triple $(\mathbb{Z}^m, \N^m, (n_0, \dots, n_{m-1}))$. Dimension groups can also be defined as those ordered abelian groups that arise as inductive limits of ordered groups of the form $(\mathbb{Z}^m, \N^m)$, with positive homomorphisms as connective maps (\cite[Proposition~7.2.8]{rll_ktheory}).

\begin{remark}\label{remark:interpret}
We consider unital ordered abelian groups $(G,G_+,u)$ as structures in the language of pointed ordered groups $\mathcal L_{K_0,u}=\{+,\leq,u\}$, where $+$ is a binary operation, $\le$ is a binary relation and $u$ is a constant symbol. Similarly, we see abelian pointed semigroups, such as $(V(A), [1_A])$ for a unital \cstar-algebra $A$, as structures of the language of semigroups with a constant symbol $u$, $\mathcal L_{V,u} = \{ + , u \}$. Given a unital \cstar-algebra $A$ and its $K_0$-group $K_0(A)$, we interpret the latter as an $\mathcal L_{K_0,u}$-structure, where $a \le^{K_0(A)} b$ if and only if $b -a \in K_0(A)_+$ and where $u^{K_0(A)} = [1_A]$. Note that the interpretation of $\le^{K_0(A)}$ is a partial order only in case $V(A)$ has the cancellation property. In general $\le^{K_0(A)}$ is just a preorder, that is $a \le b$ and $b \le a$ does not imply $a = b$.
\end{remark}
\section{From $K_0$-Groups to AF-algebras} \label{s:K0toalg}
The first part of this section is devoted to the proof of Theorem~\ref{mainthm:K0toalg}, while in the second we cover the construction leading to Theorem~\ref{mainthm:highrank}.


\subsection{A Transfer of Strategies}\label{SS.Strategies}
Given two unital AF-algebras $A$ and $B$, thanks to Theorem~\ref{thm:EFD} and Proposition~\ref{prop:PI}, we know that $K_0(A)\equiv_{\alpha} K_0(B)$ is equivalent to $\text{II}\uparrow \EFD_\alpha(K_0(A),K_0(B))$, and that $\text{II} \uparrow \PI_\alpha(A,B)$ implies $A \equiv_\alpha B$. Exploiting this, the core idea to prove Theorem~\ref{mainthm:K0toalg} is to pass through games, in particular we show how Player~II can use a winning strategy for $\EFD_{\omega \cdot \alpha}(K_0(A),K_0(B))$ to devise a winning strategy in $\PI_\alpha(A,B)$. This is done in Theorem~\ref{thm:gamesK0alg}, using Lemma~\ref{lemma:classification}.

A reader familiar with the arguments employed in \cite{Elliott.AF} will probably find too many details in the following proposition, which we still decided to include for completeness. We treat $K_0$-groups as $\mathcal L_{K_0,u}$-structures (see Remark~\ref{remark:interpret}) and unital AF-algebras as $\mathcal{L}_{\rm C^\ast}$-structures (see \S\ref{sss.metricinf}).
\begin{theorem} \label{thm:gamesK0alg}
Let $A$ and $B$ be unital AF-algebras and let $\alpha<\omega_1$. Then
\[
\text{II}\uparrow\EFD_{\omega\cdot \alpha}(K_0(A), K_0(B)) \Rightarrow \text{II}\uparrow\PI_\alpha(A,B).
\]
\end{theorem}

\begin{proof}
Let $\{ A_n \}_{n \in \N}$ and $\{ B_n \}_{n \in \N}$ be increasing sequences of finite-dimensional \cstar-algebras such that 
\[
A_0 = \mathbb{C} 1_A, \ B_0 = \mathbb{C} 1_B, \ \overline{\bigcup_{n \in \N} A_n} = A, \ \overline{\bigcup_{n \in \N} B_n} = B.
\]
For each $n \in \N$, let $\iota^A_n\colon A_n \to A$ and $\iota^B_n\colon B_n \to B$ be the unital inclusions. By~\cite[Theorem~6.3.2]{rll_ktheory}, 
\[
K_0(A) =\bigcup_{n \in \N} K_0(\iota^A_n)(K_0(A_n))\text{ and }K_0(B) = \bigcup_{n \in \N} K_0(\iota^B_n)(K_0(B_n)).
\]
We start a match of $\PI_\alpha(A,B)$, and we show how Player II has to answer Player I's moves in order to win. 

\textbf{Round $k = 0$.} Suppose that Player I plays $(\alpha_0, c_0,\e_0)$ with $c_0\in A$ (the case $c_0 \in B$ is analogous). First, Player II finds $r_0\in\N$ and $ a_0\in A_{r_0}$ such that $\norm{a_0- c_0} <\e_0$. Being finite-dimensional, the \cstar-algebra $ A_{r_0}$ is isomorphic to $\bigoplus_{\ell < m_0} M_{n_\ell}(\mathbb{C})$ for some $m_0 \in \N$. Let $\{ e_{i,j,\ell} \}_{\ell < m_0, i,j < n_\ell}$ be a system of matrix units for $A_{r_0}$. The dimension group $(K_0(A_{r_0}), K_0(A_{r_0})_+)$ is isomorphic to $(\bbZ^{m_0 }, \bbN^{m_0})$, with generators $\{ [e_{0,0, \ell}] \}_{\ell < m_0}$.

We start in parallel a match of $\EFD_{\omega \cdot \alpha}(K_0(A), K_0(B))$ with Player I opening with the following $m_0$ moves
\[
\left(\omega \cdot \alpha_0 + m_0 - 1, K_0(\iota^A_{r_0})([e_{0,0,0}])\right), \ldots, \left(\omega \cdot \alpha_0, K_0(\iota^A_{r_0})([e_{0,0, m_0 - 1}])\right).
\]
Player II answers with some $d_{0,0} \ldots, d_{0, m_0 - 1} \in K_0(B)$ according to their winning strategy for $\EFD_{\omega \cdot \alpha}(K_0(A), K_0(B))$. As a consequence, the positive unital group homomorphism
\begin{align} 
\Phi_0\colon K_0(\iota^A_{r_0})(K_0(A_{r_0})) &\to K_0(B) \label{align:iso} \\ 
K_0(\iota^A_{r_0})([e_{0,0, \ell}]) &\mapsto d_{0,\ell} \quad \forall \ell < m_0, \nonumber
\end{align}
 is an isomorphism onto its image.
 
 Consider $B_0 = \mathbb{C} 1_B \subseteq B$. Let $\Theta\colon K_0(B_0) \to K_0(A_{r_0})$ be the group homomorphism sending $[1_{B}]$ to $[1_{A_{r_0}}]$. It follows that
\[
(\Phi_0 \circ K_0(\iota^A_{r_0}))(\Theta([1_{B}])) = \Phi_0(K_0(\iota^A_{r_0})([1_{A_{r_0}}])) \stackrel{\eqref{align:iso}}{=} [1_B].
\]
This implies $(\Phi_0 \circ K_0(\iota^A_{r_0})) \circ \Theta = K_0(\iota^B_0)$, since $[1_{B}]$ generates $K_0(B_0)$. By \cite[Lemma~7.3.3]{rll_ktheory} there exists $s_0 \in \N$ and a positive unital group homomorphism $\Phi'\colon K_0( A_{r_0}) \to K_0(B_{s_0})$ making the following diagram commute
\begin{equation} \label{eq:diagram1}
 \begin{tikzcd}
 K_0(B_{s_0}) \arrow{r}{K_0(\iota^B_{s_0})} & K_0(B) \\
 K_0( A_{r_0}) \arrow{u}{\Phi'} \arrow{r}{K_0(\iota^A_{r_0})} &K_0(\iota^A_{r_0})(K_0(A_{r_0}))\arrow[swap]{u}{\Phi_0}
 \end{tikzcd}
\end{equation}

By Proposition~\ref{lemma:classification}\ref{item:ex}, there exists a unital $^*$-homomorphism $\Psi'\colon A_{r_0} \to B_{s_0}$ such that $K_0(\Psi') = \Phi'$. Note that, given $\ell < m_0$, we have
\begin{equation} \label{eq:inj}
K_0(\iota^B_{s_0})([\Psi'(e_{0,0, \ell})]) =K_0(\iota^B_{s_0})(\Phi'([(e_{0,0, \ell})]) ) \stackrel{\eqref{eq:diagram1}}{=} \Phi_0 ( K_0(\iota^A_{r_0})([e_{0,0, \ell}]))\stackrel{\eqref{align:iso}}{=} d_{0,\ell} > 0.
\end{equation}
The latter inequality follows from the fact that $\Phi_0$ is an isomorphism onto its image, and that $K_0(\iota^A_{r_0})([e_{0,0,\ell}]) = [\iota^A_{r_0}(e_{0,0,\ell})] > 0$, which is the case since $\iota^A_{r_0}$ is injective. This shows that $\Psi'(e_{0,0,\ell}) \not = 0$ for every $\ell < m_0$, and thus that $\Psi'$ is injective on each addendum $M_{n_\ell}(\mathbb{C})$ of $A_{r_0}$, and hence on $A_{r_0}$ itself (as matrix algebras are simple). Set $ b_0 := \Psi'( a_0) \in B_{s_0}$, and let it be Player II's move for round $k = 0$.

We conclude step $k=0$ by setting some notation needed in the upcoming steps of the proof. Set $C_0 := A_{r_0}$ and $D_0 := \Psi'(C_0)$. Denote by $\iota_{D_0}$ the inclusion of $D_0$ in $B$, and $\iota$ the inclusion of $D_0$ in $B_{s_0}$. Let moreover $\Psi_0\colon C_0 \to D_0$ be the co-restriction of $\Psi'$ to $D_0$. Since $K_0(\iota) \circ K_0( \Psi_0) = K_0(\Psi') = \Phi'$, and $\iota_{D_0} = \iota^B_{s_0} \circ \iota$, we can infer that
\[
K_0(\iota_{D_0}) \circ K_0(\Psi_0) =K_0(\iota^B_{s_0}) \circ K_0(\iota) \circ K_0(\Psi_0) = K_0(\iota^B_{s_0}) \circ \Phi'.
\]
In view of \eqref{eq:diagram1}, this entails that following diagram commutes
\[
\begin{tikzcd}
 K_0(D_0) \arrow{r}{K_0(\iota_{D_0})} & K_0(B) \\
 K_0(C_0) \arrow{u}{K_0(\Psi_0)} \arrow{r}{K_0(\iota^A_{r_0})} &K_0(\iota^A_{r_0})(K_0(A_{r_0}))\arrow[swap]{u}{\Phi_0}
\end{tikzcd}
\] 

\textbf{Round $k+1$.}
The general round is played similarly to round 0, but we also use Proposition~\ref{lemma:classification}\ref{item:uni} to make sure that Player II's move is coherent with their previous moves. Let's see it in detail.

Suppose that at the beginning of round $k+1 \in \N$ the game is in position
\[
(\alpha_0, c_0,\e_0,a_0, b_0), \ldots, (\alpha_{k-1}, c_{k-1},\e_{k-1}, a_{k-1}, b_{k-1}).
\]
We make the following inductive assumptions.
\begin{enumerate}[label=(\roman*)]
\item \label{item:ia1} We have naturals $r_0 \le \ldots\le r_{k-1}$ and $s_0\le \ldots\le s_{k-1} \in \N$ such that $ a_0,\ldots, a_{j} \in A_{r_{j}}$ and $ b_0, \ldots, b_{j} \in B_{s_{j}}$ for all $j<k$. Furthermore, we have natural $m_0,\dots, m_{k-1}$ such that $A_{r_j}\cong\bigoplus_{\ell<m_j}M_{n_{\ell,j}}$ for all $j < k$.
\item \label{item:ia2} We assume that there exists an ongoing match of $\EFD_{\omega \cdot \alpha}(K_0(A), K_0(B))$ which is in position
\begin{align*}
\left( \omega \cdot \alpha_0 + m_0 -1, e_{0,0}, d_{0,0} \right), &\ldots, \left( \omega \cdot \alpha_0, e_{m_0-1 ,0}, d_{m_0-1,0} \right), \\
&\ldots, \\
 \left( \omega \cdot \alpha_{k-1} + m_{k-1} -1, e_{0,k-1}, d_{0,k-1} \right), &\ldots, \left( \omega \cdot \alpha_{k-1}, e_{m_{k-1}-1,k-1}, d_{m_{k-1}-1,k -1} \right),
\end{align*}
with $\{ e_{i,j} \}_{j < k, i < m_j} \subseteq K_0(A)$, $\{ d_{i,j} \}_{j < k, i < m_j} \subseteq K_0(B)$. We assume that Player II has played their moves following a winning strategy.
\item We assume there exist two finite-dimensional subalgebras $C_{k-1} \subseteq A_{r_{k-1}}$ and $D_{k-1} \subseteq B_{s_{k-1}}$ with unital immersions $\iota_{C_{k-1}}\colon C_{k-1} \to A$ and $\iota_{D_{k-1}}\colon D_{k-1} \to B$ such that:
\begin{enumerate}[label=(iii.\alph*)]
\item $ a_0, \ldots, a_{k-1} \in C_{k-1}$ and $ b_0, \ldots, b_{k-1} \in D_{k-1}$,
\item \label{item:iso} there exists an isomorphism of \cstar-algebras $\Psi_{k-1}\colon C_{k-1} \to D_{k-1}$ mapping $ a_j \mapsto b_j$ for every $j < k$,
\item $K_0(\iota_{C_{k-1}})(K_0(C_{k-1}))$ is the subgroup of $K_0(A)$ generated by $\{ e_{i,j} \}_{j < k, i < m_j}$, and similarly $\{ d_{i,j} \}_{j < k, i < m_j}$ generate $K_0(\iota_{D_{k-1}})(K_0(D_{k-1}))$ in $K_0(B)$,
\item let $\Phi_{k-1}\colon K_0(\iota_{C_{k-1}})(K_0(C_{k-1}))\to K_0(\iota_{D_{k-1}})(K_0(D_{k-1}))$ be the positive unital isomorphism mapping $e_{i,j} \mapsto d_{i,j}$ for every $j < k$ and $i< m_j$, which exists since Player II is playing according to a winning strategy. Then the following diagram commutes
\begin{equation} \label{eq:diagramk-1}
\begin{tikzcd}
K_0(D_{k-1}) \arrow{r}{K_0(\iota_{D_{k-1}})} \arrow[swap,leftarrow]{d}{K_0(\Psi_{k-1})} & K_0(B) \arrow[leftarrow]{d}{\Phi_{k-1}} \\
K_0(C_{k-1}) \arrow{r}{K_0(\iota_{C_{k-1}})}&K_0(\iota_{C_{k-1}})(K_0(C_{k-1}))
\end{tikzcd}
\end{equation}
\end{enumerate}
\end{enumerate}

If $\alpha_{k-1} = 0$ then the game $\PI_\alpha(A,B)$ terminates, and by~\ref{item:iso} Player II wins. If, on the other hand, $\alpha_{k-1} > 0$, the game continues. Suppose that for their next move Player I picks $(\alpha_{k}, c_{k},\e_k)$, where $c_k\in A$ (again, if Player I plays $c_k \in B$ the proof is verbatim). Let $r_{k} \in \N$ greater than $r_{k-1}$ and big enough so that there is $ a_k\in A_{r_{k}}$ with $\norm{ c_k-a_k}<\e_k$. We have that $ A_{r_{k}} \cong \bigoplus_{\ell < m_{k}} M_{n_\ell}(\mathbb{C})$ for some $m_{k} \in \N$. Let $\{ f_{i,j, \ell} \}_{\ell < m_{k},i,j < n_\ell}$ be a system of matrix units for $A_{r_{k}}$. We continue the match of $\EFD_{\omega \cdot \alpha}(K_0(A), K_0(B))$ from item \ref{item:ia2} in the inductive assumption, and we suppose that Player I plays the following $m_{k}$ moves
\[
 \left( \omega \cdot \alpha_{k} + m_{k} -1 , K_0(\iota^A_{r_{k}})([f_{0,0,0}])\right), \ldots, \left( \omega \cdot \alpha_{k}, K_0(\iota^A_{r_{k}})([f_{0,0,m_{k} -1}]) \right).
\]
By item \ref{item:ia2} of the inductive assumption, Player II can keep playing according to their winning strategy for $\EFD_{\omega \cdot \alpha}(K_0(A), K_0(B))$ with some $d_{0,k}, \ldots, d_{m_{k}-1,k} \in K_0(B)$. The positive unital group homomorphism 
\begin{align} 
\Phi_{k}\colon K_0(\iota^A_{r_{k}})(K_0(A_{r_{k}})) &\to K_0(B) \label{align:isok} \\ 
K_0(\iota^A_{r_{k}})([f_{0,0,\ell}]) &\mapsto d_{\ell,k}\quad \forall \ell < m_{k}, \nonumber \\
e_{i,j} &\mapsto d_{i,j} \quad \forall j \le k, i < m_k, \nonumber
\end{align}
is therefore an isomorphism onto its image which is equal to $\Phi_{k-1}$ on $K_0(\iota_{C_{k-1}})( K_0(C_{k-1}))$.

Denote the inclusions of $C_{k-1}$ in $A_{r_{k}}$ by $\iota_1$ . By setting
\[
\Theta: =K_0(\iota_1) \circ K_0(\Psi_{k-1}^{-1})\colon K_0(D_{k-1}) \to K_0(A_{r_{k}}),
\]
and
\[
\Xi:= \Phi_{k} \circ K_0(\iota^A_{r_{k}})\colon K_0(A_{r_{k}}) \to K_0(B),
\] 
we can see that 
\begin{eqnarray*}
\Xi \circ \Theta &=& \Phi_{k} \circ K_0(\iota^A_{r_{k}}) \circ K_0(\iota_1) \circ K_0(\Psi_{k-1}^{-1}) \\
&=& \Phi_{k} \circ K_0(\iota_{C_{k-1}}) \circ K_0(\Psi_{k-1}^{-1}) \\
&=& \Phi_{k-1} \circ K_0(\iota_{C_{k-1}}) \circ K_0(\Psi_{k-1}^{-1}) \\
&\stackrel{\mathclap{\eqref{eq:diagramk-1}}}{=}& K_0(\iota_{D_{k-1}}) \circ K_0(\Psi_{k-1}) \circ K_0(\Psi_{k-1}^{-1})= K_0(\iota_{D_{k-1}}).
\end{eqnarray*}
The one before last equality holds since $\Phi_{k}$ is equal to $\Phi_{k-1}$ on $K_0(\iota_{C_{k-1}})( K_0(C_{k-1}))$. By \cite[Lemma~7.3.3]{rll_ktheory} there exists $s_{k} \in \N$ greater than $s_{k-1}$ and a positive unital group homomorphism $\tilde \Phi\colon K_0(A_{r_{k}}) \to K_0(B_{s_{k}})$ making the following diagrams commute
\begin{equation} \label{eq:diagramk}
 \begin{tikzcd}
 K_0(D_{k-1}) \arrow{r}{K_0(\iota_2)} \arrow[swap]{dr}{\Theta: =K_0(\iota_1) \circ K_0(\Psi_{k-1}^{-1})} &K_0(B_{s_{k}}) \arrow{r}{K_0(\iota^B_{s_{k}})} & K_0(B) \\
 & K_0( A_{r_{k}}) \arrow{u}{\tilde \Phi} \arrow{r}{K_0(\iota^A_{r_{k}})} &K_0(\iota^A_{r_{k}})(K_0(A_{r_{k}}))\arrow[swap]{u}{\Phi_{k}}
 \end{tikzcd}.
\end{equation}
By Proposition~\ref{lemma:classification}\ref{item:ex} there exists a unital $^*$-homomorphism $\tilde \Psi\colon A_{r_{k}} \to B_{s_{k}}$ such that $K_0(\tilde \Psi) = \tilde \Phi$. Note moreover that, given $\ell < m_{k}$, we have
\begin{eqnarray*}
K_0(\iota^B_{s_{k}})([\tilde \Psi(f_{0,0,\ell})]) &=& K_0(\iota^B_{s_{k}})(\tilde \Phi([f_{0,0,\ell}]) ) \\ &\stackrel{\mathclap{\eqref{eq:diagramk}}}{=}& \Phi_{k} ( K_0(\iota^A_{r_{k}}) ([f_{0,0,\ell}]))\\ &\stackrel{\mathclap{\eqref{align:isok}}}{=}& d_{\ell,k} > 0,
\end{eqnarray*}
therefore, by the same argument as the one after equation \eqref{eq:inj}, $\tilde \Psi$ is injective on $A_{r_{k}}$.

Denote the inclusion of $D_{k-1}$ into $B_{s_k}$ by $\iota_2$. By \eqref{eq:diagramk} we have that $K_0(\tilde \Psi \circ \iota_1 \circ \Psi_{k-1}^{-1}) = K_0(\iota_2)$. By Proposition~\ref{lemma:classification}\ref{item:uni} there exists therefore a unitary $u \in B_{s_{k}}$ such that $\text{Ad}(u) \circ \tilde \Psi \restriction C_{k-1} = \Psi_{k-1}$. Define $\Psi_{k}$ as the co-restriction of $\text{Ad}(u) \circ \tilde \Psi$ to $D_{k} := \text{Ad}(u)(\tilde \Psi(A_{r_{k}}))$. Then $\Psi_{k}\colon A_{r_{k}} \to D_{k}$ is an isomorphism mapping $ a_j \mapsto b_j$ for every $j < k$, and such that $K_0 (\iota) \circ K_0(\Psi_{k}) = \tilde \Phi$, where $\iota$ is the inclusion of $D_{k}$ into $B_{s_{k}}$. Let $ b_{k} := \Psi_{k}( a_{k})$ be Player~II's move for round $k$. Finally, set $C_{k} := A_{r_{k}}$. Arguing like at the end of round 0, it is possible to check that all inductive assumptions are satisfied.
\end{proof}

\begin{corollary} \label{cor:K0toalg}
Let $A$ and $B$ be unital AF-algebras and let $\alpha < \omega_1$. Then
\[
(K_0(A), K_0(A)_+, [1_A]) \equiv_{\omega \cdot \alpha} (K_0(B), K_0(B)_+,[1_B]) \Rightarrow A \equiv_\alpha B.
\]
\end{corollary}
\begin{proof}
By Theorem~\ref{thm:EFD}, we know that $\text{II}\uparrow\EFD_{\omega \cdot \alpha}(K_0(A), K_0(B))$. By Theorem~\ref{thm:gamesK0alg} this gives $\text{II}\uparrow\PI_\alpha(A,B)$. Finally, Proposition~\ref{prop:PI} yields the desired conclusion.
\end{proof}

\subsection{AF-algebras with arbitrarily high Scott Rank}\label{SS.groups}

In this section we build a family of simple unital AF-algebras $\{A_\alpha\}_{\alpha < \omega_1}$ such that
$A_\alpha \equiv_\alpha A_\beta$ but $A_\alpha \not \cong A_\beta$, for all $\alpha < \beta$, thus proving
 Theorem~\ref{mainthm:highrank}. Thanks to Corollary~\ref{cor:K0toalg}, this task is reduced to finding a family of simple dimension groups satisfying analogous relations. We exhibit such a family in Corollary~\ref{T.MainGps}, after having simplified further the problem to a question about linear orders.

We restrict our attention to the following class of groups.
\begin{definition}\label{DefOrd}
Let $X$ be a compact space, and set
\[
G_X := \{ f \colon X \to \mathbb Q \mid f \text{ is continuous with respect to the discrete topology on } \mathbb Q \}.
\]
We define two order relations on $G_X$. Given $g_0$ and $g_1$ in $G_X$, write
\[
 g_0\leq g_1 \iff g_0(x) \leq g_1(x), \ \forall x \in X,
\]
and
\[
g_0\ll g_1 \iff g_0=g_1 \text{ or } g_0(x) < g_1(x), \ \forall x \in X.
\]
These orders endow $G_X$ with a structure of ordered group with the following positive cones 
\[
(G_X)_{+}=\{g\in G_X\mid 0\leq g\}\text{ and }(G_X)_{+,\ll}=\{g\in G_X\mid 0\ll g\}.
\]
\end{definition}

If $X$ is metrizable, by compactness, the group $G_X$ is countable. Moreover the groups $(G_X,(G_X)_{+})$ and $(G_X,(G_X)_{+,\ll})$ are unperforated and satisfy the Riesz property, as $\mathbb Q$ satisfies these properties with both the usual order $\le$ and the strict one $\ll$. In particular, the two groups are dimension groups by Theorem~\ref{thm:dimgr}.

The constant function $1$ is an order unit in both $(G_X,(G_X)_{+})$ and $(G_X,(G_X)_{+,\ll})$, but note that $(G_X,(G_X)_{+})$ is not a simple group, even for $X=\{0,1\}$, as the identity function on $\{0,1\}$ belongs to $(G_X)_{+}$, but is not an order unit for $(G_X,(G_X)_{+})$. The main reason for considering the order $\ll$ on $G_X$ is exactly that it endows it with a structure of simple dimension group.
\begin{lemma}\label{L.dimension}
Let $X$ be a compact metrizable space. The ordered group $(G_X,(G_X)_{+,\ll})$ is a simple dimension group.
\end{lemma}
\begin{proof}
The fact that $(G_X,(G_X)_{+,\ll})$ is a dimension group is a consequence of Theorem~\ref{thm:dimgr}. To show simplicity, fix $g\in (G_X)_{+,\ll}$ and $f \in G_X$. By compactness of $X$ there are rational numbers $M > 0$ and $\e>0$ such that $\lvert f(x) \rvert < M$ and $g(x)>\e$ for all $x\in X$. Let $n\in\N$ with $n>M/\e$. Then $-ng \ll f \ll ng$. Since $f$ and $g$ are arbitrary, we are done.
\end{proof}
The next proposition shows that, in case $X$ is 0-dimensional, we can recover the topology of $X$ from either of the groups $(G_X,(G_X)_{+})$ and $(G_X,(G_X)_{+,\ll})$.

\begin{proposition}\label{PropSpaceHomeo}
Let $X$ and $Y$ be $0$-dimensional compact spaces. The following are equivalent:
\begin{enumerate}
\item\label{c.spgp1} $X$ and $Y$ are homeomorphic;
\item\label{c.spgp2} $(G_X,(G_X)_{+,\ll})\cong (G_Y,(G_Y)_{+,\ll})$;
\item\label{c.spgp3} $(G_X,(G_X)_{+})\cong (G_Y,(G_Y)_{+})$.
\end{enumerate}
\end{proposition}
\begin{proof}
That \eqref{c.spgp1} implies both \eqref{c.spgp2} and \eqref{c.spgp3} is obvious, so we are left with proving that \eqref{c.spgp2}$\Rightarrow$\eqref{c.spgp3} and that \eqref{c.spgp3}$\Rightarrow$\eqref{c.spgp1}.

\eqref{c.spgp2}$\Rightarrow$\eqref{c.spgp3}. It suffices to show that every positive group homomorphism
\[
\Phi\colon(G_X,(G_X)_{+,\ll})\to (G_Y,(G_Y)_{+,\ll})
\]
also satisfies $\Phi((G_X)_+) \subseteq (G_Y)_+$. Suppose that $g \in (G_X)_+$. If $0\leq \Phi(g)$ were false, there would exist $y \in Y$ such that $\Phi(g)(y)<0$. Choose a rational $\e>0$ such that $\Phi(g)(y) + \e\Phi(1)(y) <0$, which implies $\Phi(g+\e)(y) <0$. This means that $\Phi(g +\e) \not \gg 0$, which contradicts $g+\e \gg 0$.

\eqref{c.spgp3}$\Rightarrow$\eqref{c.spgp1}. Let $\Phi\colon (G_X,(G_X)_{+})\to (G_Y,(G_Y)_{+})$ be an isomorphism. If $g \in G_X$, let
\[
Z_g := \{ x \in X \mid g(x) = 0 \}.
\]
As $\mathbb Q$ is considered with the discrete topology, this is a clopen subset of $X$. Let $\sim_X$ be the equivalence relation on $(G_X)_+$ defined by
\[
g_0\sim_X g_1 \iff\exists n\in\N (g_0\leq ng_1\wedge g_1\leq ng_0).
\]
Equivalently, $g_0 \sim_X g_1$ if and only if $Z_{g_0} = Z_{g_1}$. On $((G_X)_+/\sim_X)$ let 
\[
[g_0]\leq_X[g_1]\iff \exists n\in\N (g_1\leq ng_0),
\]
which is equivalent to $Z_{g_0} \subseteq Z_{g_1}$. Since $\Phi$ is $\leq$ preserving, it maps $\sim_X$ equivalent pairs to $\sim_Y$ equivalent pairs. Similarly, if $[g_0]\leq_X[g_1]$ then $[\Phi(g_0)]\leq_Y[\Phi(g_1)]$. Since $\Phi$ is an isomorphism, it induces an isomorphism of ordered sets
\[
\tilde\Phi\colon ((G_X)_+/\sim_X,\leq_X)\to ((G_Y)_+/\sim_Y,\leq_Y).
\]
Noticing that $((G_X)_+/\sim_X,\leq_X)\cong (\Clop(X),\subseteq)$, where $\Clop(X)$ is the Boolean algebra of clopen subsets of $X$, and that a similar isomorphism holds for $Y$, we have that the posets $(\Clop(X), \subseteq)$ and $(\Clop(Y), \subseteq)$ are isomorphic. By Stone duality, $X$ and $Y$ are homeomorphic.
\end{proof}

Notice that the presence of clopen sets is essential in Proposition~\ref{PropSpaceHomeo}. If, for example, $X$ is connected, then all continuous functions $X\to\mathbb Q$ (where $\mathbb Q$ is considered with the discrete topology) are constant, and so in this case $G_X=\mathbb Q$. 

Prototypical examples of metrizable compact $0$-dimensional spaces are successor countable ordinals endowed with the order topology. Given $\beta < \omega_1$, let 
\[
\mathcal G_{\beta+1}:=(G_{\beta+1}, (G_{\beta+1})_{+,\ll},1),
\]
$1$ being the constant function.
\begin{notation}
If $\beta$ is an ordinal, a clopen interval partition $\{U_i\}_{i \in I}$ of $\beta+1$ is completely determined by an increasing sequence of ordinals
\[
0= \beta_0<\cdots<\beta_{k} = \beta,
\]
with $U_{0} = [\beta_0, \beta_1]$, and $U_i = [\beta_i + 1, \beta_{i+1}]$. We denote such a partition by the ordered set $\{\beta_0 , \ldots, \beta_{k} \}$. Inclusions between $\{\beta_i\}_{i\leq k}$ and $\{\gamma_j\}_{j\leq k'}$ correspond to refinements of the associated clopen interval partitions.

If $\{\beta_0,\ldots,\beta_k\}$ is an interval partition of $\beta+1$, we let $\mathcal{S} (\{\beta_0, \ldots ,\beta_k \}) \subseteq \mathcal G_{\beta+1}$ be the subgroup of all functions which are constant on the intervals $[\beta_i+1,\beta_{i+1}]$. $\mathcal{S} (\{\beta_0, \ldots ,\beta_k \} )$ is isomorphic to $\mathbb Q^k$, and each element of $\mathcal G_{\beta+1}$ belongs to $\mathcal S(\{\beta_0,\ldots,\beta_{k}\})$ for some interval partition $\{\beta_0,\ldots,\beta_k\}$.
\end{notation}

In what follows, we interpret ordinals as linear orders in the language with a binary relation $\{<\}$. Unital dimension groups are considered as pointed ordered groups with respect to the signature $\mathcal{L}_{K_0, u}$ introduced in Remark \ref{remark:interpret}.

In analogy with Corollary~\ref{cor:K0toalg}, the next proposition shows that in order to study the relation $\equiv_\alpha$ on the group $\mathcal G_{\beta +1}$, it is enough to understand elementary equivalence on the linear orders $\beta +1$. The idea is the same as in Theorem~\ref{thm:gamesK0alg}. Since $\equiv_\alpha$ can be characterized in terms of Dynamic EF-games, by using Theorem~\ref{thm:EFD} twice, it suffices to show that if Player II knows how to win $\EFD_{{\omega}\cdot{\alpha}}(\beta+1,\gamma+1)$, then they can devise a winning strategy for $\EFD_{\alpha}(\mathcal G_{\beta+1}, \mathcal G_{\gamma+1})$.
	
\begin{proposition}\label{prop:Transfer}
For all countable ordinals $\alpha$, $\beta$, and $\gamma$, we have
\[
\beta+1 \equiv_{{\omega}\cdot{\alpha}} \gamma + 1 \Rightarrow 	\mathcal G_{\beta+1} \equiv_\alpha \mathcal G_{\gamma+1}. 
\]
\end{proposition}

\begin{proof}
By Theorem~\ref{thm:EFD} used twice, it is enough to show that if $\text{II} \uparrow \EFD_{\omega\cdot\alpha}(\beta+1,\gamma+1)$ then $\text{II} \uparrow\EFD_\alpha(\mathcal G_{\beta+1}, \mathcal G_{\gamma+1})$.

Suppose Player I plays $(\alpha_0,g_0) \in \alpha \times \mathcal G_{\beta+1}$ (the case where Player I plays $(\alpha_0,g_0) \in \alpha \times \mathcal G_{\gamma+1}$ is analogous). Let $\{ \beta^0_0, \ldots, \beta^0_{m_0} \}$ be an interval partition of $\beta+1$ such that $g_0 \in \mathcal{S} (\{\beta^0_0,\ldots ,\beta^0_{m_0} \} )$. We start an auxiliary match of $\EFD_{\omega \cdot \alpha}(\beta+1,\gamma+1)$ with Player~I playing the following $m_0 +1$ moves
\[
({\omega}\cdot{\alpha_0}+m_0,\beta^0_0), \ldots, ({\omega}\cdot{\alpha_0},\beta^0_{m_0}).
\]
Since $\text{II} \uparrow \EFD_{\omega\cdot\alpha}(\beta+1,\gamma+1)$, Player II answers to Player I's moves according to their strategy, with some $\gamma^0_0 , \ldots , \gamma^0_{m_0} < \gamma +1$. As a consequence the map $\beta_i^0\mapsto \gamma_i^0$ defines an isomorphism of linear orders which in turn induces an isomorphism of unital ordered groups
\[
\iota_0\colon \mathcal{S} (\{\beta^0_0,\ldots ,\beta^0_{m_0} \} )\to \mathcal{S} (\{\gamma^0_0, \ldots ,\gamma^0_{m_0} \} ).
\]
Player II's move for round $k = 0$ of the ongoing match of $\EFD_{\alpha}(\mathcal G_{\beta+1}, \mathcal G_{\gamma+1})$ is $h_0 :=\iota_0(g_0)$.

The successive rounds are played analogously. If round $k$ of $\EFD_\alpha(\mathcal G_{\beta+1},\mathcal G_{\gamma+1})$ has been played with moves
\[
(\alpha_0, g_0, h_0), \ldots, (\alpha_k,g_k,h_k),
\]
then there are two families of ordinals $\{\beta_j^{i}\}_{i\leq k, j\leq m_i}$ and $\{\gamma_j^{i}\}_{i\leq k, j\leq m_i}$ such that the map $\beta_j^i\mapsto\gamma_j^i$ is an order isomorphism (these are obtained from the auxiliary match of $\EFD_{\omega \cdot \alpha}(\beta+1,\gamma+1)$). This map induces an isomorphism 
\[
\iota_k\colon \mathcal S(\{\beta_j^{i}\}_{i\leq k, j\leq m_i})\to \mathcal S(\{\gamma_j^{i}\}_{i\leq k, j\leq m_i}),
\]
which maps 1 to 1 and is such that $\iota_k(g_j) = h_j$ for every $j \le k$. Now suppose $k$ is such that $\alpha_k =0$. The isomorphism $\iota_k$ witnesses that Player II has won this match of $\EFD_\alpha(\mathcal G_{\beta+1},\mathcal G_{\gamma+1})$.
\end{proof}	

The following result is due to C.\ Karp (see also \cite[Theorem~7.26]{van_games}).

\begin{theorem}[{{\cite[Theorem~3]{karp}}}]\label{ThKarp}
If $\varepsilon$ is an ordinal satisfying $\varepsilon = \omega^\varepsilon$, then for every ordinal $\delta\neq 0$ we have $\varepsilon \equiv_\varepsilon \varepsilon\cdot \delta$.
\end{theorem}

 We let $(\varepsilon_\alpha)_{\alpha<\omega_1}$ denote the continuous increasing sequence enumerating all countable ordinals $\varepsilon$ that satisfy $\omega^\varepsilon = \varepsilon$.

\begin{corollary}\label{T.MainGps}
The family $\{ \mathcal G_{\varepsilon_\alpha+1}\}_{\alpha<\omega_1}$ is such that
\[
\mathcal G_{\varepsilon_\alpha+1} \not \cong \mathcal G_{\varepsilon_\beta+1} \text{ and } \mathcal G_{\varepsilon_\alpha+1} \equiv_{\varepsilon_\alpha} \mathcal G_{\varepsilon_\beta+1}, \ \forall
\alpha < \beta < \omega_1.
\]
\end{corollary}

\begin{proof}
Note that the ordinals $(\varepsilon_\alpha)_{\alpha < \omega_1}$ are multiplicatively indecomposable, that is\linebreak $\gamma \cdot \varepsilon_\alpha = \varepsilon_\alpha$ for every $\gamma < \varepsilon_\alpha$ and every $\alpha < \omega_1$. Fix $\alpha < \beta < \omega_1$. The previous observation along with Theorem~\ref{ThKarp} implies that $\varepsilon_\alpha \equiv_{\varepsilon_{\alpha}} \varepsilon_\beta$. This also entails that $\varepsilon_\alpha + 1 \equiv_{\varepsilon_{\alpha}} \varepsilon_\beta+1$. Since moreover ${\omega}\cdot{\varepsilon_\alpha} = \varepsilon_{\alpha}$, Proposition~\ref{prop:Transfer} allows to infer that
\[
\mathcal G_{\varepsilon_\alpha+1} \equiv_{\varepsilon_\alpha} \mathcal G_{\varepsilon_\beta+1}.
\]
To finish the proof, by Proposition~\ref{PropSpaceHomeo} it remains to check that $\varepsilon_\alpha+1$ and $\varepsilon_\beta+1$ are not homeomorphic whenever $\alpha< \beta$. This can be checked using a Cantor-Bendixson argument (or by directly invoking either \cite{kieftenbeld} or \cite[Proposition~1.6.6]{goncharov}). 
\end{proof}
 
After having reduced the problem to a statement about linear orders, we climb back to \cstar-algebras and we obtain the following.

\begin{corollary} \label{corollary:scott}
Given $\alpha < \omega_1$, let $A_\alpha$ be the simple unital AF-algebra whose dimension group is $\mathcal G_{\varepsilon_\alpha +1}$. Then $\{ A_\alpha\} _{\alpha < \omega_1}$ is a collection such that $A_\alpha \not \cong A_\beta$ and $A_\alpha \equiv_\alpha A_\beta$ for every $\alpha < \beta < \omega_1$.
\end{corollary}
\begin{proof}
Lemma~\ref{L.dimension} ensures that each $A_\alpha$ is simple. Remember that $A_\alpha \cong A_\beta$ if and only if $\mathcal G_{\varepsilon_\alpha +1} \cong \mathcal G_{\varepsilon_\beta +1}$, by the Elliott Classification Theorem~in \cite{Elliott.AF}. Hence, by Corollary~\ref{T.MainGps}, it follows that $A_\alpha \not \cong A_\beta$ for every $\alpha < \beta$. The part about elementarily equivalence follows by Corollary~\ref{T.MainGps} and Corollary~\ref{cor:K0toalg}, since $\omega \cdot \varepsilon_\alpha = \varepsilon_\alpha \geq \alpha$.
\end{proof}

We refer the reader to Appendix \ref{s:Bratteli} for an explicit computation of the Bratteli diagrams of these algebras.

\section{From \cstar-Algebras to $K_0$-Groups}\label{S.AlgtoK0}

The goal of this section is to prove a partial converse to Corollary~\ref{cor:K0toalg}, and to understand how much (infinitary) information regarding $K_0$ can be extracted from (infinitary) information about a given \cstar-algebra.

Let $\mathcal L_{\mathrm{C}^*,c} = \{+, \cdot, ^*, \{ \lambda\}_{\lambda \in \mathbb{C}},0, c\}$ be the language of pointed \cstar-algebras with a constant~$c$. Fix a non-zero minimal projection $q\in\mathcal K$. If $A$ is a unital \cstar-algebra, we interpret canonically $A\otimes\mathcal K$ as an $\mathcal L_{\mathrm{C}^*,c}$-structure by setting $c^{A\otimes K}=1_A\otimes q$.\footnote{The language $\mathcal L_{\mathrm{C}^*,c}$ is formally speaking exactly the same as $\mathcal L_{\mathrm{C}^*}$ from \S\ref{sss.metricinf}. Nevertheless, we decided to use different notations since we use the latter for unital \cstar-algebras, interpreting $u$ as 1, while the former is reserved to non-unital \cstar-algebras of the form $A \otimes \mathcal K$.} As for $V(A)$ (resp., $K_0(A)$), we will always view it as an $\mathcal L_{V,u}$ (resp., $\mathcal L_{K_0,u}$)-structure (see Remark~\ref{remark:interpret}).

The following is a generalization to $\mathcal L_{\omega_1\omega}$ of \cite[Theorem~3.11.1]{Muenster}, where the same statement is proved for first order logic.

\begin{theorem}\label{thm:algbstoK_0}
Let $A$ and $B$ be unital \cstar-algebras. If $\alpha<\omega_1$ and $A\otimes\mathcal K\equiv_{\omega + 2\cdot \alpha +2} B\otimes\mathcal K$, then $(K_0(A), K_0(A)_+,[1_A]) \equiv_{\alpha}(K_0(B), K_0(B)_+,[1_B])$. In particular, if $\alpha < \omega_1$ is a limit ordinal with $\alpha \geq \omega^2$ then
\[A\otimes\mathcal K\equiv_\alpha B\otimes\mathcal K \Rightarrow (K_0(A), K_0(A)_+,[1_A]) \equiv_{\alpha}(K_0(B), K_0(B)_+,[1_B]).\]
\end{theorem}

We start by introducing some preliminary definitions and notation needed to prove Theorem~\ref{thm:algbstoK_0}. Let $n$ be a positive integer and let $\mathcal{B} := \{ e_i \}_{i < n}$ be the orthonormal basis of $\mathbb{C}^n$
\[
e_i = (0,\ldots,\underbrace{1}_i,\ldots,0).
\]
Whenever $A$ is a \cstar-algebra, we will write $A_1$ for the unit ball of $A.$
Let $\delta > 0$ and fix, once and for all, a finite subset $\{ b^{(n,\delta)}_h \}_{ h < J}$ in the unit ball of $\mathbb{C}^n$ such that for every $b \in \mathbb{C}^n_1$ there exists $h < J$ such that
\begin{equation} \label{eq:bh}
\| b - b^{(n, \delta)}_h \| < \delta/n.
\end{equation}
Without loss of generality we can assume that $\mathcal{B}$ is contained in such set. For every $h < J$ let $\{ \lambda^{(n,\delta,h)}_i \}_{i < n}$ be the coordinates of $b^{(n,\delta)}_h$ with respect to the basis $\mathcal{B}$, that is
\[
b_h = \sum_{i < n} \lambda^{(n,\delta,h)}_i e_i, \ \forall h < J.
\]
For $h_0, h_1 <J$ we define the following $^*$-polynomials in the variables $\bar x = (x_0, \dots, x_{n-1})$
\begin{align*}
P_{h_0}(\bar x) &:= \sum_{i < n} \lambda^{(n,\delta,h_0)}_i x_i, \\
P_{h_0, h_1}(\bar x) &:= \left( \sum_{i < n} \lambda^{(n,\delta,h_0)}_i x_i \right)\left( \sum_{i < n} \lambda^{(n,\delta,h_1)}_i x_i \right) - \sum_{i < n} \lambda^{(n,\delta,h_0)}_i\lambda^{(n,\delta,h_1)}_i x_i .
\end{align*}
Finally, define the $\mathcal{L}_{\mathrm{C}^\ast, c}$-formulas
\[
\phi_{n, \delta}(\bar x) := \max_{h_0,h_1 < J} \{ \| P_{h_0,h_1}(\bar x) \|, \| P_{h_0}(\bar x) \| \dotminus 1 \}, \forall n \in \N 	\setminus \{0 \}, \delta > 0,
\]
where $x \dotminus y$ is an abbreviation for $\max \{ 0, x- y \}$.

The following lemma is a consequence of Ulam stability for finite-dimensional \cstar-algebras (see \cite{ulam_mv}). The crucial aspect in the statement of the lemma is the fact that $\delta > 0$ can be chosen uniformly over all positive $n$.
\begin{lemma}\label{lemma:ulamstable}
For every $\e>0$ there is $\delta>0$ such that for every \cstar-algebra $A$ and all positive $n$, if $\bar a$ is an $n$-tuple of positive contractions in $A$ such that
\begin{equation} \label{eq:ai}
\phi_{n, \delta}(\bar a) < \delta,
\end{equation}
then there exist $b_0, \dots, b_{n-1} \in A$ such that $\| a_i - b_i \| < \e$ for all $i < n$ and such that
\[
\phi_{n, \delta}(\bar b) = 0.
\]
In particular $b_0, \dots, b_{n-1}$ are orthogonal projections.
\end{lemma}

\begin{proof}
Take $\bar a =(a_0, \dots, a_{n-1})$ as in the statement of the lemma. We claim that the map, extended by linearity,
\begin{align*}
\Phi \colon \mathbb{C}^n &\to A \\
e_i &\mapsto a_i
\end{align*}
is a $\delta'$-$^*$-homomorphism in the sense of \cite[Definition 1.1]{ulam_mv}, with $\delta' = 5 \delta + 4 \delta^2$. Granted this, the conclusion of the lemma follows by \cite[Theorem~A]{ulam_mv}.

Note that the map $\Phi$ is linear by definition, and it is $^*$-preserving since each $a_i$ is positive. We need to show that the remaining conditions in \cite[Definition 1.1]{ulam_mv} are satisfied, that is
\begin{equation} \label{eq:bound}
\| \Phi(a) \| \le 1 + \delta', \ \forall a \in \mathbb{C}^n_1,
\end{equation}
\begin{equation} \label{eq:mult}
\| \Phi(a) \Phi(b) - \Phi(ab) \| \le \delta', \ \forall a,b \in \mathbb{C}^n_1.
\end{equation}
Notice that both inequalities are automatically satisfied with $\delta$ in place of $\delta'$ if $a,b \in \{b^{(n,\delta)}_h \}_{h < J}$, because of \eqref{eq:ai}. In order to prove them for general $a,b$, note first that
\[
\| \Phi(a) - \Phi(b) \| \le n \| a - b \|, \ \forall a,b \in \mathbb{C}^n_1,
\]
hence, by our choice of $\{b^{(n,\delta)}_h \}_{h < J}$ (see \eqref{eq:bh}), for every $a \in \mathbb{C}^n_1$ there exists $h < J$ such that $\| \Phi(a) - \Phi(b_h) \| < \delta$. This in particular entails
\[
\| \Phi(a) \| \le \| \Phi(b_h) \| + \|\Phi(a) - \Phi(b_h) \| \le 1 + \delta + \delta \le 1 + \delta', \ \forall a \in \mathbb{C}^n_1.
\]
Similarly, given $a,b \in \mathbb{C}^n_1$ and $h_0, h_1 < J$ such that $\| \Phi(a) - \Phi(b_{h_0}) \| < \delta$ and $\| \Phi(b) - \Phi(b_{h_1}) \| < \delta$, we have
\begin{align*}
\| \Phi(a) \Phi(b) - \Phi(ab) \| &\le \| \Phi(b_{h_0}) \Phi(b_{h_1}) + \Phi(b_{h_0}b_{h_1}) \| + 2\delta(1 +2 \delta) + 2 \delta \\
& \le \delta + 2\delta(1 +2 \delta) + 2 \delta \\
& = \delta'.\qedhere
\end{align*}
\end{proof}

The idea behind the proof of Theorem~\ref{thm:algbstoK_0} is to provide a translation from $\mathcal L_{K_0,u}$-formulas to $\mathcal L_{\mathrm{C}^*,c}$ ones. We do this by passing through $\mathcal L_{V,u}$-formulas. Given a tuple of projections $\bar p \in \mathcal{P}(A\otimes\mathcal K)^n$, we denote by $[\bar p]$ the tuple of the corresponding classes in $V(A)$.

\begin{proposition} \label{prop:V(A)}
For every $\mathcal L_{V,u}$-formula $\phi(\bar x)$ there exists an $\mathcal L_{\mathrm{C}^*,c}$-formula $\psi(\bar z)$ of the same arity and of rank at most $\omega + \rk(\phi(\bar x)) $ such that for every unital \cstar-algebra $A$
\[
V(A) \models \phi([\bar p]) \text{ if and only if }\psi^{A \otimes \mathcal{K}}(\bar p) = 0.
\]
In particular, for all unital \cstar-algebras $A$ and $B$, 
\[
A\otimes \mathcal K\equiv_{\omega + \alpha}B\otimes \mathcal K\Rightarrow V(A) \equiv_\alpha V(B).
\]
\end{proposition}

\begin{proof}
The proof is by induction on the rank of $\phi(\bar x)$. We prove a stronger statement that will make the inductive step work, specifically the step taking care of the infinitary connectives. More specifically, we show that there exists $\delta > 0$ such that for every $\mathcal L_{V,u}$-formula $\phi(\bar x)$ there is a $1$-Lipschitz $\mathcal L_{\mathrm{C}^*,c}$-formula $\psi(\bar z)$ with $\rk(\psi(\bar z))\leq \omega + \rk(\phi(\bar x))$ satisfying, for every unital \cstar-algebra $A$ and tuple of projections $\bar p$ of the appropriate length
\begin{enumerate}[label=(\roman*)]
\item\label{V(A)-cond1} $V(A) \models \phi([\bar p])$ if and only if $\psi^{A \otimes \mathcal{K}}(\bar p) = 0$,
\item\label{V(A)-cond2} $\psi^{A \otimes \mathcal{K}}(\bar p) < \delta$ implies $\psi^{A \otimes \mathcal{K}}(\bar p) = 0$.
\end{enumerate}
We start with the atomic case, so we assume that
\[
\phi(\bar x, \bar y) := \sum_{i<n_0} x_i + m_0 \cdot u = \sum_{j<n_1} y_j + m_1 \cdot u.
\]
Since the proof works the same way for every $m_0,m_1 \in \mathbb{N}$, we assume $m_0=m_1=1$. Recall that the constant $c$ in the language $\mathcal L_{\mathrm{C}^*,c}$ is interpreted as $1\otimes q$ in $A\otimes\mathcal K$, where $q\in\mathcal K$ is a minimal projection. A standard argument in \cstar-algebras shows that there exists $1 > \gamma > 0$ such that, if $B$ is a \cstar-algebra, $p_0, p_1 \in B$ are projections and $s \in B$ is a contraction satisfying $\max \{ \| s^*s - p_0 \|, \| ss^* -p_1 \| \} < \gamma$, then $p_0 \sim p_1$, that is there exists a partial isometry $r \in B$ such that $r^*r = p_0$ and $rr^* = p_1$. Let $\delta' > 0$ be the value obtained from Lemma~\ref{lemma:ulamstable} for $\e = \gamma/2$, and set $\delta = \min \{ \gamma/2, \delta' \}$.
Consider the formula 
\begin{eqnarray*}
\psi(\bar z,\bar w)&:=&\inf_v\inf_{v_0, \dots , v_{n_0}}\inf_{u_0, \dots, u_{n_1}} \max_{i< n_0, j< n_1} \{ \norm{v_i^*v_i-z_i},\norm{u_j^*u_j-w_j},\\
&& \norm{v_{n_0}^*v_{n_0}-c},\norm{u_{n_1}^*u_{n_1}-c},\norm{vv^*-\sum\nolimits_{i \le n_0} v_iv_i^*},\norm{v^*v-\sum\nolimits_{j\le n_1} u_ju_j^*},\\
&& \phi_{n_0 + n_1 + 2,\delta}(v_0v_0^*, \dots, v_{n_0}v_{n_0}^*,u_0u_0^*, \dots, u_{n_1}u_{n_1}^*) \}.
\end{eqnarray*}
Note that $\psi(\bar z, \bar w)$ has finite rank and that it is $1$-Lipschitz in the free variables $(\bar z, \bar w)$.

We start by verifying condition~\ref{V(A)-cond2}. Let us unravel $\psi$: suppose that we have sequences of projections $\bar p$ and $\bar q$ from $A \otimes \mathcal{K}$, for any given unital \cstar-algebra $A$, such that $\psi^{A \otimes \mathcal{K}}(\bar p,\bar q) < \delta$. By Lemma~\ref{lemma:ulamstable} there are pairwise orthogonal projections $\{ p_i',q_j' \}_{i \le n_0, j \le n_1} \subseteq A\otimes \mathcal{K}$ and $r_0, \dots, r_{n_0}, s_0, \dots, s_{n_1} \in A \otimes \mathcal K$ such that, setting $p_{n_0} = q_{n_1} = 1 \otimes q$,
\[
\max_{i \le n_0, j \le n_1}
\{ \norm{ r_i^*r_i - p_i}, \norm{r_ir_i^* - p_i'}, \norm{ s_j^*s_j - q_j }, \norm{s_js_j^* - q_j'} \} < \frac{\gamma}{2},
\]
which, by the choice of $\gamma$ gives
\[
p_i \sim p_i' \text{ and } q_j \sim q_j', \ \forall i \le n_0, j\le n_1.
\]
The inequality $\psi^{A \otimes \mathcal{K}}(\bar p,\bar q) < \delta$ also entails the existence of $t \in A \otimes \mathcal{K}$ which satisfies
\[
\max \Big\{ \Big \lVert tt^* - \sum\nolimits_{i \le n_0} p_i' \Big \rVert, \Big \lVert t^*t - \sum\nolimits_{j \le n_1} q_j' \Big \rVert \Big\} < \gamma,
\]
and therefore also the sums $\sum_{i \le n_0} p_i'$ and $\sum_{j \le n_1} q_j'$ are Murray von Neumann equivalent. In other words, $\psi^{A \otimes \mathcal{K}}(\bar p,\bar q)=0$, and moreover the value 0 is a minimum attained by a family of partial isometries witnessing $\sum[p_i]+[1\otimes q]=\sum[q_j]+[1\otimes q]$, or, equivalently, $V(A) \models \phi([\bar p], [\bar q])$. This also proves the right-to-left implication in condition \ref{V(A)-cond1}. The other implication in \ref{V(A)-cond1} is immediate since $\psi$ formalizes the notion of $\sim$ when computed in $A\otimes\mathcal K$.

Consider the case $\phi(\bar x) = \neg \phi_0(\bar x)$. By inductive assumption there exists an $\mathcal L_{\mathrm{C}^*,c}$-formula $\psi_0(\bar z)$ of rank at most $\omega + \rk(\phi(\bar x))$ such that for every unital \cstar-algebra $A$ the formula $\psi_0(\bar z)$ satisfies conditions~\ref{V(A)-cond1} and \ref{V(A)-cond2}. The formula $\psi(\bar z) := \delta \dotminus \psi_0(\bar z)$ is as required, being obtained by a combination of 1-Lipschitz connectives.

Let $\phi(\bar x) = \bigwedge_{n \in \N} \phi_n(\bar x)$. By inductive assumption there are $\mathcal L_{\mathrm{C}^*,c}$-formulas $\psi_n(\bar z)$ for $n \in \N$ of rank at most $\omega + \rk(\phi(\bar x)) $ such that for every unital \cstar-algebra $A$ the $1$-Lipschitz formulas $\psi_n(\bar z)$ satisfy conditions~\ref{V(A)-cond1} and \ref{V(A)-cond2}. Given some $M \in \N$ greater than $\delta$, the formulas $\min \{ \psi_n(\bar z), M \}$ are again $1$-Lipschitz $\mathcal L_{\mathrm{C}^*,c}$-formulas whose range is contained in $[0,M]$; hence we can assume that all $\psi_n(\bar z)$ have the same range. We set $\psi(\bar z) := \bigvee_{n \in \N} \psi_n(\bar z)$. Since all the $\psi_n(\bar y)$'s are $1$-Lipschitz, so is $\psi(\bar z)$.

Finally, consider the case $\phi(\bar x) = \exists y \phi_0(\bar x, y)$. By inductive assumption there exists an $\mathcal L_{\mathrm{C}^*,c}$-formula $\psi_0(\bar z, w)$ of rank at most $\omega + \rk(\phi_0(\bar x,y))$ such that for every unital \cstar-algebra $A$ the formula $\psi_0(\bar z,w)$ satisfies conditions~\ref{V(A)-cond1} and \ref{V(A)-cond2}, and it is $1$-Lipschitz. It is a well-known fact that, given some $a \in B$ of norm one in a given \cstar-algebra $B$, for every $\e > 0$ there exists $\delta' > 0$ such that $\max \{ \| a - a^* \|, \| a- a^2 \| \} < \delta'$ implies the existence of a projection $p \in B$ such that $\| a - p \| < \e$ (see e.g. \cite[Exercise 2.7]{rll_ktheory}). Direct computations show that for $\e < 1/2$, setting $\delta' = \e/3$ does the job. Indeed, in this case, with $b = (a + a^*) /2$, we have that $\| a- b \| < \e/6$, and that $\| b - b^2 \| < 5\e /6$. The latter in particular implies that the spectrum of $b$ is contained in $[-5\e /6, 5\e/6] \cup [1 - 5\e/6, 1 + 5\e/6]$, and by continuous functional calculus there is a projection $p \in B$ such that $\| b -p \| < 5\e/6$, and thus $\| a - p\| < \e$. With these observations in mind, and exploiting the fact that $\psi_0(\bar z, w)$ is 1-Lipschitz in $w$, it can be checked that the formula
\[
\psi(\bar z) := \inf_w \psi_0(\bar z, w) + 3 \cdot {\max \{ \| w - w^* \| , \| w - w^2 \| \}},
\]
is the required $\mathcal L_{\mathrm{C}^*,c}$-formula. Since $\psi(\bar z)$ is furthermore $1$-Lipschitz in $\bar z$, we are done.
\end{proof}

Given an abelian semigroup $H$, let $G(H)$ be its Grothendieck group. The elements of $G(H)$ are pairs of elements in $H$ such that $(g_0, g_1) = (h_0, h_1)$ if and only if there is $k \in H$ satisfying $g_0 + h_1 + k = h_0 + g_1 + k$. Let $G(H)_+ = \{ (g,0) \mid g \in H \}$ and define a preorder relation on $G(H)$ by setting $g \le h$ if and only if $h-g \in G(H)_+$. With these definitions, the Grothendieck group of a pointed abelian semigroup $(H,v)$ can be naturally interpreted as an $\mathcal{L}_{K_0, u}$-structure by setting $u^{G(H)} = (v,0)$, in analogy to what we did in Remark \ref{remark:interpret}.

Given a tuple of elements $\bar g :=(g_0 = (g_{0,0}, g_{0,1}),\dots, g_{n-1}=(g_{n-1,0},g_{n-1,1})) \in G(H)^n$, we denote by $\bar g_j$, for $j = 0,1$, the tuple $(g_{0,j}, \dots, g_{n-1,j}) \in H^n$. Below, we generalize the fact that the first order theory of $H$ is interpretable in the first order theory of $G(H)$ to the $\mathcal L_{\omega_1\omega}$-setting.

\begin{proposition} \label{prop:semigp}
For every $\mathcal L_{K_0,u}$-formula $\phi(\bar x)$ there exists an $\mathcal L_{V,u}$-formula $\psi(\bar y)$ whose arity is double the arity of $\phi$, whose rank is at most $2 \cdot \rk(\phi(\bar x)) +2$ and such that, for every abelian pointed semigroup $H$ and every tuple $\bar g$ from $G(H)$, we have that 
\[
G(H) \models \phi(\bar g) \text{ if and only if } H \models \psi(\bar g_0, \bar g_1).
\]
In particular, if $H_0$ and $H_1$ are abelian pointed semigroups and $\alpha$ is a countable ordinal, then
\[
H_0 \equiv_{2 \cdot \alpha +2} H_1 \Rightarrow G(H_0) \equiv_\alpha G(H_1).
\]
If in addition $\alpha$ is limit, then
\[
H_0 \equiv_\alpha H_1 \Rightarrow G(H_0) \equiv_\alpha G(H_1).
\]
\end{proposition}

\begin{proof}
We prove this by induction on the complexity of the formula. If $\phi(\bar x)$ is atomic then, up to reindexing the variables, it has two possible forms
\[
\sum_{j=0}^{m-1} x_j + m_0 \cdot u = \sum_{j=m}^{n-1} x_j + m_1 \cdot u,
\]
or
\[
\sum_{j=0}^{m-1} x_j + m_0 \cdot u \le \sum_{j=m}^{n-1} x_j + m_1 \cdot u.
\]
In the first case $\psi(\bar y)$ can be chosen to be
\[
\psi(\bar y) := \exists z\left(\sum_{j=0}^{m-1}y_{j,0} + \sum_{j=m}^{n-1} y_{j,1} + m_0 \cdot u + z =
\sum_{j=0}^{m-1}y_{j,1} + \sum_{j=m}^{n-1} y_{j,0} + m_1 \cdot u + z\right).
\]
In the second case we set
\[
\psi(\bar y) := \exists z,w\left(\sum_{j=0}^{m-1}y_{j,0} + \sum_{j=m}^{n-1} y_{j,1} + m_0 \cdot u + z +w =
\sum_{j=0}^{m-1}y_{j,1} + \sum_{j=m}^{n-1} y_{j,0} + m_1 \cdot u + z\right).
\]
The rest of the proof follows by an induction argument, and we explicitly check only the existential case. Let $\phi(\bar x) = \exists z \phi_0(\bar x, z)$. By inductive hypothesis there exists a formula $\psi_0(\bar y_0, \bar y_1, w_0, w_1)$ of rank at most $2 \cdot \rk(\phi_0(\bar x, z)) + 2$ such that for every abelian pointed semigroup $H$ and every tuple $(\bar g, g')$ in $G(H)$ we have
\[
G(H) \models \phi_0(\bar g, g')\text{ if and only if } H \models \psi_0(\bar g_0, \bar g_1, g'_0, g'_1).
\]
The formula $\exists w_0, w_1 \psi_0(\bar y_0, \bar y_1, w_0, w_1)$ is as required, as it has rank at most $2 \cdot \rk(\phi_0(\bar x, z)) + 4 = 2 \cdot \rk(\phi(\bar x)) + 2$.\\
The second statement in the proposition is now clear.
Finally, assume $\alpha$ is a countable limit ordinal and $H_0 \equiv_\alpha H_1.$ To prove that $G(H_0) \equiv_\alpha G(H_1),$ it suffices to show that $G(H_0) \equiv_\beta G(H_1),$ for every $\beta < \alpha.$ This follows because $2 \cdot \beta + 2 < \alpha$ for every $\beta < \alpha.$
\end{proof}

\begin{proof}[Proof of Theorem~\ref{thm:algbstoK_0}]
If $A\otimes\mathcal K\equiv_{\omega + 2\cdot \alpha +2} B\otimes\mathcal K$, then $V(A) \equiv_{2 \cdot \alpha +2}V(B)$ by Proposition~\ref{prop:V(A)}. By Proposition~\ref{prop:semigp} this gives $(K_0(A), K_0(A)_+,[1_A])\equiv_\alpha (K_0(B), K_0(B)_+,[1_B])$. If $\alpha \geq \omega^2$ is limit and 
$A\otimes\mathcal K\equiv_\alpha B\otimes\mathcal K,$ then $V(A) \equiv_\alpha V(B)$ by  Proposition~\ref{prop:V(A)}, because $\alpha = \omega + \alpha.$ By Proposition~\ref{prop:semigp}, then $(K_0(A), K_0(A)_+,[1_A])\equiv_\alpha (K_0(B), K_0(B)_+,[1_B])$.
\end{proof}

It is natural to ask whether, for a unital \cstar-algebra $A$, one could extract the $\mathcal L_{\omega_1\omega}$-theory of $A\otimes\mathcal K$ from that of $A$. The following question summarizes this problem.

\begin{question}\label{ques:compacts}
Let $\alpha<\omega_1$. Is there $\beta<\omega_1$ such that for all unital \cstar-algebras $A$ and $B$, if $A\equiv_\beta B$ then $A\otimes\mathcal K\equiv_\alpha B\otimes\mathcal K$?
\end{question}

In an optimal situation the way of obtaining the $\beta$ required to answer positively Question~\ref{ques:compacts} is uniform over $\alpha$ (we would be tempted to conjecture that $\omega \cdot \alpha$ works).

An idea to answer Question~\ref{ques:compacts} positively could be to follow the proof of Proposition~\ref{prop:V(A)} and try to find, for any $\mathcal L_{\mathrm{C}^*,c}$-sentence $\phi$, another $\mathcal L_{\mathrm{C}^*}$-sentence $\psi$ such that $\phi^{A\otimes\mathcal K}=\psi^A$ and such that the rank of $\psi$ can be bounded (uniformly) by the rank of $\phi$ (here, as $A$ is unital, $A$ is canonically an $\mathcal L_{\mathrm{C}^*}$-structure). The problem is that we cannot ensure that the formulas constructed in this manner have all the same modulus of uniform continuity (as they involve larger and larger matrix amplifications), and therefore the inductive procedure halts when taking infinitary connectives.

Using the Partial Isomorphism games, we provide a partial answer to Question~\ref{ques:compacts} in Proposition~\ref{P.compacts}. We view this result as evidence of the validity of the approach to infinitary logic via games. 

The following is well-known to operator algebraists. We record it for completeness.

\begin{lemma}\label{L.eps}
Let $A$ be a \cstar-algebra, $n\in\N$ and $\e>0$. Then there is $\delta>0$ such that for all elements $a_{i,j}$ and $a'_{i,j}$ of $A$, where $i,j< n$, if 
\[
\norm{a_{i,j}-a_{i,j}'}<\delta,
\]
then, for $(a_{i,j}), (a_{i,j}') \in M_n(A)$, 
\[
\norm{(a_{i,j})-(a_{i,j}')}<\e.
\]
\end{lemma}

\begin{proof}
It is an exercise to check that $\| (a_{i,j}) \| \le \sum_{i,j < n} \| a_{i,j} \|$, hence it suffices to take $\delta = \e/n^2$.
\end{proof}

Given a \cstar-algebra $A$, in the following proposition we naturally identify $M_n(A)$ with a subalgebra of $A \otimes \mathcal{K}$.
\begin{proposition}\label{P.compacts}
Let $A$ and $B$ be \cstar-algebras. Suppose that $\text{II} \uparrow \PI_{\omega\cdot\alpha}(A,B)$. Then $\text{II} \uparrow \PI_\alpha(A\otimes\mathcal K,B\otimes\mathcal K)$.
\end{proposition}

\begin{proof}
Suppose that in round $k = 0$ Player I plays $(\alpha_0,c_0,\e_0)$ where $\alpha_0<\alpha$, $\e_0>0$ and $c_0\in A \otimes \mathcal{K}$ (in case Player I plays $c_0 \in B \otimes \mathcal{K}$ the proof is verbatim). Player II finds $n_0$ and $a_0\in M_{n_0}(A)$ such that $\norm{c_0-a_0}<\e_0/2$, and they then write $a_0$ as a matrix $(a_{0,i,j})_{i,j < n_0}$ where each $a_{0,i,j}\in A$. Let moreover $\delta_0$ be the value given by Lemma~\ref{L.eps} for $\e = \e_0/2$. Player II starts then a round of $\PI_{\omega\cdot\alpha}(A,B)$ as if Player I had played already the first $n_0^2$ moves 
\[
((\omega\cdot\alpha_0+i\cdot n_0+j, a_{0,i,j},\delta_0))_{i,j < n_0}.
\]
Player II answers with their winning strategy for $\PI_{\omega\cdot\alpha}(A,B)$, and gets elements\linebreak $(\tilde a_{0,i,j},\tilde b_{0,i,j})\in A\times B$ with $\norm{\tilde a_{0,i,j}-a_{0,i,j}}<\delta_1$, for $i,j< n_0$. Finally, they play the pair 
\[
((\tilde a_{0,i,j}),(\tilde b_{0,i,j}))\in M_{n_0}(A)\times M_{n_0}(B).
\]
 By the choice of $\delta_0$, we have that 
\[
\norm{c_0-(\tilde a_{0,i,j})}< \norm{(a_{0,i,j})-(\tilde a_{0,i,j})}+\norm{c_0-a_0}<\e_0.
\]
Player II plays the rest of the game similarly, using the winning strategy they have for $\PI_{\omega \cdot \alpha}(A,B)$ to decide their moves in the ongoing match of $\PI_\alpha(A \otimes \mathcal{K}, B \otimes \mathcal{K})$. The game terminates with a sequence
\[
((\alpha_0, c_0, \e_0, \tilde{a}_0, \tilde{b}_0), \dots, (\alpha_{k-1}, c_{k-1}, \e_{k-1}, \tilde a_{k-1}, \tilde b_{k-1})),
\]
where $\alpha_{k-1} = 0$ and $(\tilde a_\ell, \tilde b_\ell) = ((\tilde a_{\ell,i,j}), (\tilde b_{\ell,i,j})) \in M_{n_\ell}(A) \times M_{n_\ell}(B)$, for $\ell < k$, and such that the sequence
\[
((\tilde a_{\ell,i,j}, \tilde b_{\ell,i,j}))_{\ell < k, i,j < n_\ell},
\]
is obtained following Player II's winning strategy for $\PI_{\omega \cdot \alpha}(A,B)$. Therefore, the map $\tilde a_{\ell,i,j}\mapsto \tilde b_{\ell,i,j}$, for $\ell< k, i,j< n_\ell$, gives an isomorphism between $\mathrm{C}^*(\{\tilde a_{\ell,i,j}\}_{\ell < k, i,j < n_\ell})$ and $\mathrm{C}^*(\{\tilde b_{\ell,i,j}\}_{\ell < k, i,j < n_\ell})$, which naturally induces an isomorphism between $\mathrm{C}^*(\{ \tilde a_\ell\}_{\ell < k})$ and $\mathrm{C}^*(\{\tilde b_\ell\}_{\ell < k})$, thus proving the proposition.
\end{proof}

Finally, we deduce the following Corollary~\ref{thm:combined}, which is of a summarizing nature.
In this statement we use the notation $A \equiv_{ \alpha}^{\PI} B$ (instead of $\text{II}\uparrow\PI_\alpha(A,B)$) for the statement that Player II has a winning strategy for $\PI_\alpha(A,B)$. We so far refrained from using such notation as there is no a priori reason that $\equiv_{ \alpha}^{\PI}$ is actually an equivalence relation, but this is now justified through Corollary~\ref{thm:combined} itself.   
\begin{corollary}\label{thm:combined}
	Let $A$ and $B$ be unital AF-algebras, let $\alpha<\omega_1$ be a non-zero ordinal which is such that $\alpha = \omega^\omega \cdot \beta,$ for some $\beta  \leq \alpha.$  Then the following are equivalent

\begin{enumerate}
	\item\label{el.eq1} $A \equiv_{ \alpha}^{\PI} B$;
	\item\label{el.eq2} $A\otimes\mathcal K\equiv_\alpha B\otimes\mathcal K$;
	\item\label{el.eq3} $(K_0(A),K_0(A)_+,[1_A]) \equiv_{ \alpha} (K_0(B),K_0(B)_+,[1_B]).$
\end{enumerate}
\end{corollary}
\begin{proof}
	\eqref{el.eq1}$\Rightarrow$\eqref{el.eq2}.
From $A \equiv_{ \alpha}^{\PI} B$ follows $\text{II} \uparrow \PI_\alpha(A\otimes\mathcal K,B\otimes\mathcal K)$ by Proposition~\ref{P.compacts} and because $\omega \cdot \alpha = \alpha.$ Then $A\otimes\mathcal K \equiv_\alpha  B\otimes\mathcal K$ follows by Proposition~\ref{prop:PI}. 

	\eqref{el.eq2}$\Rightarrow$\eqref{el.eq3}.
$A\otimes\mathcal K \equiv_\alpha  B\otimes\mathcal K$ in turn implies $(K_0(A),K_0(A)_+,[1_A]) \equiv_{ \alpha} (K_0(B),K_0(B)_+,[1_B])$ using Theorem~\ref{thm:algbstoK_0}.

The implication \eqref{el.eq3}$\Rightarrow$\eqref{el.eq1} follows immediately from Theorem~\ref{thm:EFD} and Theorem~\ref{thm:gamesK0alg} (again using $\omega \cdot \alpha = \alpha$).
\end{proof}

\appendix
\section{Computing the Bratteli Diagrams} \label{s:Bratteli}
In this subsection we aim to give an alternative description of the ordered unital groups $\mathcal G_{\alpha+1}$ which, being dimension groups, are inductive limits of ordered groups of the form $(\mathbb Z^n,\mathbb Z^n_+)$. Our goal is to describe such inductive systems, thus getting insights on the Bratteli diagram of the associated AF-algebras.


Let
\begin{equation} \label{eq:sequence}
	G_0 \stackrel{\phi_{0,1}}{\longrightarrow} G_1 \stackrel{\phi_{1,2}}{\longrightarrow} \dots
	\stackrel{\phi_{n-1,n}}{\longrightarrow} G_n \stackrel{\phi_{n,n+1}}{\longrightarrow} \dots
\end{equation}
be an inductive limit of ordered groups. Recall that the inductive limit in the category of ordered abelian groups is the group $(G, G_+)$, where $(G, \phi_n\colon G_n \to G)$ is the inductive limit in the category of abelian groups, and $G_+$ is equal to $\bigcup_{n \in \N} \phi_n((G_{n})_+)$ (\cite[Proposition~6.2.6]{rll_ktheory}).

\subsubsection{Case $\mathcal G_k$, for $k \in \N$}
Fix a positive $k\in\N$. In what follows we will build an inductive system of ordered groups
\begin{equation}\label{eq:system}
	\bbZ^k \stackrel{\phi_{0,1}}{\longrightarrow} \bbZ^k \stackrel{\phi_{1,2}}{\longrightarrow} \dots
	\stackrel{\phi_{n-1,n}}{\longrightarrow} \bbZ^k \stackrel{\phi_{n,n+1}}{\longrightarrow} \dots
\end{equation}
and homomorphisms $\theta_n: \bbZ^k \to \bbQ^k$ such that, for every $n \in \N$
\begin{enumerate}[label=(\roman*)]
	\item $\theta_{n+1} \circ \phi_{n,n+1} = \theta_n$,
	\item $\theta_n$ is injective,
	\item $\bigcup_{n \in \N} \theta_n(\bbZ^k) = \bbQ^k$,
	\item $\bigcup_{n \in \N} \theta_n(\bbZ^k_+) = \bbQ^k_{+,\ll}$.
\end{enumerate}
This will ensure that $(\bbQ^k, \bbQ^k_{+,\ll})$ is the inductive limit of the system \eqref{eq:system} in the category of ordered groups.

The construction starts by fixing a suitable sequence of positive integers $\{ a_n \}_{n \in \N}.$ For simplicity, we will work here with the sequence $\{ a_n \}_{n \in \N}$ that is defined by the following recursive relation
\begin{equation} \label{eq:an}
	a_0:= k +2, \quad a_n := 2(a_{n-1} +1) - k +1\quad (n\geq 1),
\end{equation}
however the reader can easily check that the construction will work just as well for any other sequence $\{ a_n \}_{n \in \N}$ that exhibits some simple conditions on its growth behavior and the parity of its terms.

For every $n \in \N$, let $A_n\in M_k(\mathbb N)$ be given by
\[
A_n = 
\begin{pmatrix}
	a_n & 1 & \cdots & 1 \\
	1 & a_n & \cdots & 1 \\
	\vdots & \vdots & \ddots & \vdots \\
	1 & 1 & \cdots & a_n 
\end{pmatrix}
\]
and let $\theta_n\colon \bbZ^k \to \bbQ^k$ be the homomorphisms associated to the matrices
\[
C_n := \frac{1}{a_n!} A_n.
\]
$A_n$ is invertible and its inverse is given by
\[
A^{-1}_n = \frac{1}{a_n^2 +a_n(k -2) - k +1}
\begin{pmatrix}
	a_n + k -2 & -1 & \cdots & -1 \\
	- 1 & a_n + k-2 & \cdots & -1 \\
	\vdots & \vdots & \ddots & \vdots \\
	-1 & -1 & \cdots & a_n + k -2
\end{pmatrix}.
\]
Set
\[
b_n := \frac{a_{n+1}!}{a_n! (a^2_{n+1} +a_{n+1}(k -2) - k +1)}.
\]
Define the matrix
\begin{equation} \label{eq:1}
	B_{n,n+1} := C_{n+1}^{-1}C_n,
\end{equation}
which can be concretely computed as
\[
b_n
\begin{pmatrix}
	a_n(a_{n+1} + k -2) - k +1 & a_{n+1} - a_n & \cdots & a_{n+1} - a_n \\
	a_{n+1}- a_n & a_n(a_{n+1} + k -2) - k +1& \cdots & a_{n+1} - a_n \\
	\vdots & \vdots & \ddots & \vdots \\
	a_{n+1} - a_n & a_{n+1} - a_n & \cdots & a_n(a_{n+1} + k -2) - k +1
\end{pmatrix}.
\]
We would like to define the maps $\phi_{n,n+1}\colon \bbZ^k \to \bbZ^k$ composing the inductive system in \eqref{eq:system} as the group homomorphisms induced by $B_{n,n+1}$.

In order to do this, we need to ensure that $b_n$ is an integer for every $n \in \N$. We have that
\begin{align*}
	b_n = \frac{a_{n+1}!}{a_n! (a^2_{n+1} +a_{n+1}(k -2) - k +1)} &=\frac{ (a_n +1) (a_n +2) \dots(a_{n+1} -1) a_{n+1}}{(a_{n+1} -1)(a_{n+1} + k -1)}\\
		 &= \frac{ (a_n +1) (a_n +2) \dots (a_{n+1} -2) a_{n+1}}{(a_{n+1} + k -1)}.
\end{align*}
By our choice of $\{a_n \}_{n \in \N}$, we have $a_{n+1} + k -1 = 2(a_n +1)$ and $a_{n+1} - a_n > 4$. This allows to simplify the denominator in the expression of $b_n$ above, which is therefore an integer.

As all matrices $B_{n,n+1}$ have positive entries, the homomorphisms
$\phi_{n,n+1} \colon \bbZ^k \to \bbZ^k$ are positive and similarly it is immediate to check from the definition of $C_n$ that $\theta_n(\bbZ^k_+) \subseteq \bbQ^k_{+,\ll}$.
Furthermore, the following triangle commutes by \eqref{eq:1}
\[
\begin{tikzcd}
	\bbZ^k \arrow{r}{\phi_{n,n+1}} \arrow[swap]{dr}{\theta_n} & 
	\bbZ^k \arrow{d}{\theta_{n+1}} \\
	& \bbQ^k
\end{tikzcd}.
\]
All maps $\theta_n\colon \bbZ^k \to \bbQ^k$ are injective since the matrices $C_n$ are invertible, so all we need to check to conclude is that $\bigcup_{n \in \N} \theta_n (\bbZ^k) = \bbQ^k$ and that $\bigcup_{n \in \N} \theta_n (\bbZ^k_+)= \bbQ^k_{+,\ll}$.

For the first equality let $(p_0, \dots, p_{k-1}) \in \bbQ^k$, which we can assume is of the form $(\frac{x_0}{a_n!}, \dots, \frac{x_{k-1}}{a_n!})$ for some $\bar x:= (x_0, \dots, x_{k-1}) \in \bbZ^k$ and $n$ big enough. Since the sequence $\{ a_n \}_{n \in \N }$ has been chosen so that each $b_n$ is integer, the vector
\begin{equation} \label{eq:y}
	\bar y := \frac{1}{a_n!} C_{n+1}^{-1} \bar x = b_n (a_{n+1}^2 +a_{n+1}(k -2) - k +1) A_{n+1}^{-1} \bar x
\end{equation}
belongs to $\bbZ^k$. It follows then that
\[
\theta_{n+1}(\bar y) = C_{n+1} \bar y = \frac{1}{a_n!}C_{n+1} C_{n+1}^{-1} \bar x = \frac{1}{a_n!} \bar x,
\]
as required.

To see that $\bigcup_{n \in \N} \theta_n (\bbZ^k_{+}) = \bbQ^k_{+,\ll}$, write again an arbitrary non-zero element $(p_0, \dots, p_{k-1})$ of $\bbQ^k_{+,\ll}$ in the form $(\frac{x_0}{a_n!}, \dots, \frac{x_{k-1}}{a_n!})$ for some $\bar x:= (x_0, \dots, x_{k-1}) \in \bbZ^k$ with strictly positive entries. Notice that the entries of $\bar y$ in \eqref{eq:y} have the form
\begin{align*}
	y_0 &= b_n ((a_{n+1} + k -2) x_0 - x_1 - \dots -x_{k-1}),\\
	y_1 &= b_n (- x_0 + (a_{n+1} + k -2)x_1 - \dots -x_{k-1}),\\
	&\dots\\
	y_{k-1} &= b_n (- x_0 - x_1 - \dots +(a_{n+1} + k -2)x_{k-1}).
\end{align*}
Therefore, since all $x_j$ are assumed to be strictly positive and since the sequence $\{ a_n \}_{n \in \N}$ is strictly increasing, there is $m \ge n \in \N$ such that
\[
a_{m+1} > \max \left\{ \frac{\sum_{j \not = 0} x_j}{x_0}, \dots, \frac{\sum_{j \not = k-1} x_j}{x_{k-1}} \right\} -k +2.
\]
As a consequence, the vector $\bar y' \in \bbZ^k$ defined as
\begin{align*}
	y'_0 &= b_m ((a_{m+1} + k -2) x_0 - x_1 - \dots -x_{k-1}),\\
	y'_1 &= b_m (- x_0 + (a_{m+1} + k -2)x_1 - \dots -x_{k-1}),\\
	&\dots\\
	y'_{k-1} &= b_m (- x_0 - x_1 - \dots +(a_{m+1} + k -2)x_{k-1}),
\end{align*}
is such that all its entries are positive, and such that $\theta_{m+1}( \bar y') = \frac{1}{a_m!} \bar x$, hence $\theta_{m+1}(\frac{a_m!}{a_n!} \bar y' ) = \frac{1}{a_n!} \bar x$ as required.

The multiplicity of the arrows in the Bratteli diagram of (the AF-algebra associated to) $\mathcal{G}_k$ can be directly obtained by looking at the entries of the matrices $B_{n,n+1}$.

\subsubsection{Case $\mathcal G_{\omega +1}$}
We have that $\mathcal G_{\omega+1}$ is the inductive limit of the following system
\[
\mathcal G_1 \stackrel{\psi_{1,2}}{\longrightarrow} \mathcal G_2 \stackrel{\psi_{2,3}}{\longrightarrow} \dots \stackrel{\psi_{n-1,n}}{\longrightarrow}
\mathcal G_n \stackrel{\psi_{n,n+1}}{\longrightarrow} \dots \longrightarrow \mathcal G_{\omega+1},
\]
where the map $\psi_{n,n+1}\colon \bbQ^n \to \bbQ^{n+1}$ is the positive homomorphism defined as
\[
(p_1, \dots, p_n) \xmapsto{\psi_{n,n+1}} (p_1, \dots, p_n, p_n).
\] 
For every $k \in \N$, let $\{ \bbZ^k, \phi^{(k)}_{n,n+1} \}_{n,m \in \N}$ be the inductive system defined in the previous section whose limit is $\mathcal G_k$. We have that
\begin{equation} \label{eq:bigdia}
	\begin{tikzcd}[row sep=large]
		\mathcal G_1 \arrow[r, "\psi_{1,2}"]& \mathcal G_2 \arrow[r,"\psi_{2,3}"]&\dots \arrow[r, "\psi_{k-1,k}"] & \mathcal G_k \arrow[r, "\psi_{k,k+1}"]&\dots \arrow[r]& \mathcal G_{\omega+1} \\
		\vdots \arrow[u]& \vdots \arrow[u] &\dots & \vdots \arrow[u]& \dots \\
		\bbZ \arrow[u, "\phi^{(1)}_{2,3}"]& \bbZ^2 \arrow[u, "\phi^{(2)}_{2,3}"] &\dots & \bbZ^k\arrow[u, "\phi^{(k)}_{2,3}"]& \dots \\
		\bbZ \arrow[u,"\phi^{(1)}_{1,2}"]& \bbZ^2 \arrow[u, "\phi^{(2)}_{1,2}"] &\dots & \bbZ^k\arrow[u, "\phi^{(k)}_{1,2}"]& \dots
	\end{tikzcd}.
\end{equation}
Let $D_k$ be the $(k+1) \times k$ matrix associated to $\psi_{k,k+1}$ that is
\[
D_k = 
\begin{pmatrix}
	1 & 0 & \cdots & 0 \\
	0 & 1 & \cdots & 0 \\
	\vdots & \vdots & \ddots & \vdots \\
	0 & 0 & \cdots & 1 \\ 
	0 & 0 & \cdots & 1
\end{pmatrix}.
\]
Let $\theta_n^{(k)}\colon \bbZ^k \to \mathcal G_k$ be the maps defined in the previous section and let $C_n^{(k)}$ be the matrices associated to them. Set $\bar y := D_1 C^{(1)}_1 1_\bbZ$. By construction there is $n_2 \in \N$ big enough so that $\bar y \in \theta_{n_2}^{(2)}(\bbZ^2_+)$. Let $\eta_{1,2}\colon \bbZ \to \bbZ^2$ be the positive homomorphism associated to the matrix $E_{1,2}:= (C^{(2)}_{n_2})^{-1} D_1 C^{(1)}_1$. The map $\eta_{1,2}$ sends $\bbZ_{+}$ to $\bbZ^2_{+}$ since $1_\bbZ$ is mapped to some element in $\bbZ^2_{+}$. Similarly let $\eta_{k,k+1}\colon \bbZ^k \to \bbZ^{k+1}$ be the map associated to the matrix $E_{k,k+1}:=(C^{(k+1)}_{n_{k+1}})^{-1} D_k C^{(k)}_{n_k}$, where the sequence $\{n_k \}_{k \in \N}$ is chosen so that $D_kC^{(k)}_{n_k} \bar x\in \theta_{n_{k+1}}^{(k)}(\bbZ^{k+1}_+)$ for every $\bar x \in \bbZ^{k}_+$. Note that such sequence of integers exists since $\bbZ^{k}_+$ is finitely generated.

It follows that $\mathcal G_{\omega +1}$ is the inductive limit of the system
\[
\bbZ^1 \stackrel{\eta_{1,2}}{\longrightarrow} \bbZ^2 \stackrel{\eta_{2,3}}{\longrightarrow} \dots \stackrel{\eta_{n-1,n}}{\longrightarrow}
\bbZ^n \stackrel{\eta_{n,n+1}}{\longrightarrow} \dots,
\]
and that multiplicity of the arrows in the Bratteli diagram of (the AF-algebra associated to) $\mathcal{G}_{\omega +1}$ can be recovered from the entries of the matrices $E_{n,n+1}$.

\subsubsection{Case $\mathcal G_{\alpha +1}$ for $\alpha$ infinite ordinal}
We proceed by induction on $\alpha$.

If $\alpha = \beta +1$ where $\beta$ is an infinite successor ordinal, then $\beta $ and $\alpha$ are homeomorphic, thus $\mathcal G_{\alpha+1} \cong \mathcal G_{\beta+1}$ and the Bratteli diagram of (the AF-algebra associated to) $\mathcal G_{\alpha+1}$ can be recovered from the Bratteli diagram of $\mathcal G_{\beta+1}$.

In case $\beta$ is a limit ordinal, then let $\{ \alpha_n \}_{n \in \N}$ be an increasing sequence of ordinals with $\sup_n \alpha_n = \beta$. We have
\[
\mathcal G_{\alpha_1+1}\stackrel{\psi_{1,2}}{\longrightarrow} \mathcal G_{\alpha_2+1}\stackrel{\psi_{2,3}}{\longrightarrow} \dots \stackrel{\psi_{n-1,n}}{\longrightarrow}
\mathcal G_{\alpha_n+1} \stackrel{\psi_{n,n+1}}{\longrightarrow} \dots \longrightarrow \mathcal G_{\beta+1},
\]
where the maps $\psi_{n,n+1} \colon \mathcal G_{\alpha_n+1}\to \mathcal G_{\alpha_{n+1}+1}$ are the positive homomorphisms
\[
\psi_{n,n+1}(f)(\beta) =
\begin{cases}
	f(\beta) \text{ if } \beta \le \alpha_n, \\
	f(\alpha_{n}) \text{ if } \alpha_n < \beta \le \alpha_{n+1}.
\end{cases}
\]
We can therefore write a diagram similar to \eqref{eq:bigdia}, using the inductive system corresponding to each $\mathcal G_{\alpha_n+1}$. From here, the Bratteli diagram of $\mathcal G_{\beta+1}$ can be obtained in a similar way to the case $\beta = \omega$.

\begin{acknow}
 AVa is supported by European Union's Horizon 2020 research and innovation programme under the Marie Sk\l{}odowska-Curie grant agreement No.~891709. 
BV and AVi are supported by an Emergence en recherche grant from the Universit\'e Paris Cit\'e (IdeX), and AVi is partial supported by the ANR grant AGRUME (ANR-17-CE40-0026).
BDB is supported by a Cofund MathInParis PhD fellowship, this project has received funding from the European Union’s Horizon 2020 research and innovation programme under the Marie Sk\l{}odowska-Curie grant agreement No.~754362.
\includegraphics[width=0.55truecm]{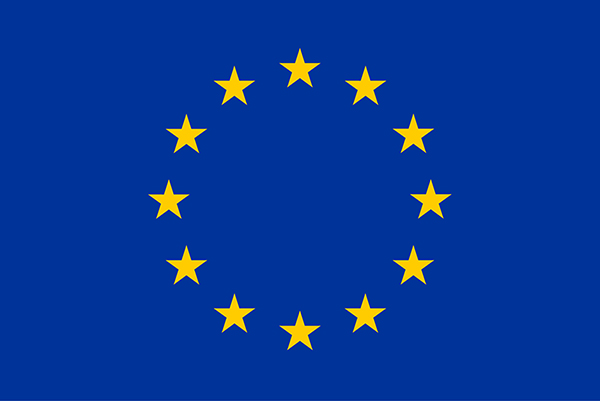}

\end{acknow}

\bibliographystyle{amsplain}
\bibliography{Bibliography}
\end{document}